\documentclass[9pt, amsfonts]{amsart}
\usepackage {amssymb, adjustbox, enumerate, amsbsy, stmaryrd}
\usepackage {amsmath}
\usepackage[mathscr]{eucal}
\usepackage{amsthm}
\usepackage{graphicx}
\usepackage {amscd}
\usepackage {epic}
\usepackage[alphabetic]{amsrefs}
\usepackage{color}

\newcommand{\spec}[0]{\operatorname{Spec}}

\newcommand{\Aut}[0]{\operatorname{Aut}}

\newcommand{\sing}[0]{\operatorname{sing}}

\newcommand{\mM}{{\mathcal{M}}}

\newcommand{\tO}{{\widetilde{\Omega}}}

\newcommand{\tOt}{{\widetilde{\Omega}_{\cX_t}}}
\newcommand{\tOo}{{\widetilde{\Omega}_{\cX_0}}}

\newcommand{\kss}{{\rm kss}}
\newcommand{\kps}{{\rm kps}}
\newcommand{\la}{{\langle}}
\newcommand{\ra}{{\rangle}}
\newcommand{\rdiv}{{\rm div}}
\newcommand{\cL}{{\mathcal{L}}}
\newcommand{\mU}{{\mathcal{U}}}
\newcommand{\mV}{{\mathcal{V}}}

\newcommand{\tcX}{{\widetilde{\mathcal{X}}}}
\newcommand{\cZ}{{\mathcal{Z}}}
\newcommand{\pdJd}{{\frac{\sqrt{-1}}{2\pi}\partial\bar{\partial}}}
\newcommand{\ku}{{\mathfrak{u}}}
\newcommand{\tku}{\widetilde{\ku}}
\newcommand{\kU}{{\mathfrak{U}}}

\newcommand{\cX}{{\mathcal{X}}}
\newcommand{\cY}{\mathcal{Y}}

\newcommand{\mY}{{\mathcal{Y}}}
\newcommand{\mL}{{\mathcal{L}}}
\newcommand{\mX}{{\mathcal{X}}}

\newtheorem{thm}{Theorem}[section]

\newtheorem{lem}[thm]{Lemma}
\newtheorem{cor}[thm]{Corollary}

\newtheorem{prop}[thm]{Proposition}

\theoremstyle{definition}
\newtheorem{defn}[thm]{Definition}

\newtheorem{rem}[thm]{Remark}

\newtheorem*{ack}{Acknowledgments}      

\newtheorem{defn-thm}[thm]{Definition--Theorem}  
\newtheorem{defn-lem}[thm]{Definition--Lemma}  

\newcommand{\QQ}{\mathbb{Q}}

\newcommand{\PP}{\mathbb{P}}

\newcommand{\CC}{\mathbb{C}}

\newcommand{\sL}{\mathscr{L}}
\newcommand{\sO}{\mathscr{O}}
\newcommand{\sI}{\mathscr{I}}
\newcommand{\sJ}{\mathscr{J}}

\newcommand{\cM}{\mathcal{M}}
\newcommand{\cU}{\mathcal{U}}
\newcommand{\cV}{\mathcal{V}}

\newcommand{\reg}{\mathrm{reg}}

\newcommand{\ti}{\widetilde}

\newcommand{\aut}{\mathrm{Aut}}
\newcommand{\faut}{\mathfrak{aut}}

\newcommand{\chow}{\mathrm{Chow}}
\newcommand{\hilb}{\mathrm{Hilb}}

\newcommand{\CM}{\mathrm{CM}}

\newcommand{\SL}{\mathrm{SL}}

\newcommand{\WP}{\mathrm{WP}}
\newcommand{\Red}{}
\newcommand{\Blue}{}
\newcommand{\lam}{\lambda}
\newcommand{\Lam}{\Lambda}

\theoremstyle{remark}

\newcommand{\sddb}{{\sqrt{-1}\partial\bar{\partial}}}

\input{xy}
\xyoption{all}

\begin{document}

\title{Quasi-projectivity of the moduli space of smooth K\"{a}hler-Einstein Fano manifolds}

\author{Chi Li}
\address{Mathematics Department, Stony Brook University\\ Stony Brook NY, 11794-3651\\ USA}
\email{chi.li@stonybrook.edu}

\author{Xiaowei Wang}
\address{Department of Mathematics and Computer Sciences\\
           Rutgers University, Newark NJ 07102-1222\\ USA}
\email{xiaowwan@rutgers.edu}

\author{Chenyang Xu}
\address{Beijing International Center of Mathematics Research\\ 5 Yiheyuan Road, Haidian District, Beijing, 100871, China}
\email{cyxu@math.pku.edu.cn}

\date{\today}
\maketitle
\begin{abstract}

In this note, we prove that there is a canonical continuous Hermitian metric on the CM line bundle over the proper moduli space $\overline{\mM}$ of smoothable K\"{a}hler-Einstein Fano varieties. The curvature of this metric is the Weil-Petersson current, which exists as a positive (1,1)-current on $\overline{\mM}$ and extends the canonical Weil-Petersson current on the moduli space parametrizing smooth K\"{a}hler-Einstein Fano manifolds $\mM$.   As a consequence, we show that the CM line bundle is nef and big on $\overline{\mM}$ and its restriction on $\mM$ is ample.
\end{abstract}

\tableofcontents

\section{Introduction}
The study of moduli spaces of polarized varieties is a  fundamental topic in algebraic geometry. The most classical case is
the moduli space of Riemann surfaces of genus$\ge 2$, whose compactification is an Deligne-Mumford stack admitting a projective coarse moduli space. People have
been trying to generalize this picture to higher dimensions, leading to the development of KSBA compactification of moduli space of canonically polarized varieties (see \cite{Kollar13}). In \cite{Vie95}, Viehweg proved a deep result that the moduli space of polarized manifolds with nef canonical line bundles is quasi-projective.
Building on the fundamental work of \cite{Kollar90} and the development of Minimal Model Program, it is proved in \cite{Fuj12} that the KSBA compactification is projective.

On the other hand, there are negative results concerning the projectivity of moduli spaces. Koll{\'a}r \cite{Kollar06} showed that the moduli space of polarized manifolds may not be quasi-projective. In particular, he proved that any toric variety can be
a moduli space of polarized uniruled manifolds. The quasi-projectivity of moduli is still open for polarized manifolds that are not uniruled and whose canonical line bundles are not nef.

Differential geometric methods have also played important roles in studying moduli space of complex manifolds. For examples, the moduli space $\mM_g$ has been studied using Teichm\"{u}ller spaces equipped with Weil-Petersson metrics, and the moduli spaces of Calabi-Yau manifolds and its Weil-Petersson metrics were studied by Tian and Todorov. Moreover, Tian's results in \cite{Tian1987} imply that there is a Hermtian line bundle on the moduli space of Calabi-Yau manifolds whose curvature form is  the Weil-Petersson metric.
Fujiki-Schumacher later \cite{FS90} considered the more general case of moduli space of K\"{a}hler manifolds admitting constant scalar curvature K\"{a}hler (cscK) metrics. They proved that the natural Weil-Petersson metric is always K\"{a}hler by interpreting it as the Chern curvature of  a determinant line bundle equipped with Quillen metric.
As a consequence, it was proved in \cite{FS90} that any compact subvariety in the moduli space of cscK manifolds with discrete automorphisms is projective.
However, all the cases considered above  requires the fiberation to be smooth, which is not the case in general. In \cite{Tian1997} Tian studied similar determinant line bundles in a singular setting and introduced the notion of CM line bundle, which will be denoted by $\lambda_{\rm CM}$ from now on in this paper. 

 It follows from Koll\'{a}r's negative result, for uniruled manifolds extra constraints \Red{must be imposed} in order for the moduli space to be quasi-projective/projective. However, Tian's study of CM line bundle and Fujiki-Schumacher's results suggest  that the moduli space of manifolds admitting canonical metrics could be quasi-projective. In this paper, we  confirm this speculation for the moduli space of Fano K\"{a}hler-Einstein (KE)  manifolds, \Blue{which  was first conjectured by Tian in \cite{Tian1997}.}

Fano manifolds in dimension 2 are called del Pezzo surfaces. It was proved in \cite{Tian1990} that a smooth del Pezzo surface admits a K\"{a}hler-Einstein metric if and only if its automorphism is reductive. Recently, based on the study of degenerations of smooth K\"{a}hler-Einstein del Pezzo surfaces in \cite{Tian1990}, proper moduli spaces of smoothable K\"{a}hler-Einstein del Pezzo varieties were constructed in \cite{OSS}. Moreover, it is shown in \cite{OSS} that these proper moduli spaces are actually projective except possibly for the case of del Pezzo surfaces of degree 1.

The higher dimensional generalization of the results in \cite{Tian1990} and \cite{OSS} was made possible thanks to the celebrated solutions to  the Yau-Tian-Donaldson conjecture (\cite{CDS1}, \cite{CDS2}, \cite{CDS3}, \cite{Tian2014}). The moduli space of higher dimensional smooth K\"{a}hler-Einstein Fano manifolds, denoted by $\mathcal{M}$
from now on, was studied in \cite{Tian2012}, \cite{Don2013}, \cite{Odaka14a}. 
More recently, a \Blue{proper algebraic compactification $\overline{\mM}$ of $\mM$ was constructed in \cite{LWX} (see also \cite{Odaka14b}). It is further believed that $\overline{\mM}$ should be projective (see \cite{OSS}, \cite{LWX}, \cite{Odaka14b}).} This paper is a step towards establishing this. The main technical result of this paper is the following descent and extension result.
\begin{thm}\label{thmext}
The CM line bundle $\lambda_{\rm CM}$ descends to a line bundle $\Lambda_{\rm CM}$ on the proper moduli space $\overline{\mM}$. Moreover, there is a canonically defined continuous Hermitian metric $h_{\rm DP}$ on $\Lambda_{\rm CM}$ whose curvature form is a positive current $\omega_{\rm WP}$ on $\overline{\mM}$ which extends the canonical Weil-Petersson current $\omega_{\rm WP}^{\circ}$ on $\mM$. 
\end{thm}
Note that although we use the notion of current on singular complex spaces defined in Definition \ref{defcur},
$\omega_{\WP}^{\circ}$ is actually a smooth K\"{a}hler-metric on a dense open set $\mM'$ of $\mM$ (see Section \ref{WP-CM} and the proof of Theorem \ref{CMemb} in Section 6). So equivalently, we can say that $\omega_{\WP}$ extends the canonical smooth K\"{a}hler metric $\omega_{\rm WP}^\circ|_{\mM'}$ on $\mM'$.

With $(\Lambda_{\rm CM}, h_{\rm DP})$ at hand (or equivalently, Weil-Petersson current $\omega_{\rm WP}$ with controlled behavior), we can apply a quasi-projective criterion as Theorem \ref{NM} to get the following result.
\begin{thm}\label{CMemb}
$\Lambda_{\rm CM}$ is nef and big over $\overline{\mM}$. Moreover, for the normalization morphism $n: \overline{\mM}^{\rm n} \to \overline{\mM}$ which induces an isomorphism over $\mM$, the rational map $\Phi_{|n^*(m\Lambda_{\rm CM})|} $ embeds $\mM$ into $\mathbb{P}^{N_m-1}$ for $m\gg 1$ with $N_m=\dim H^0(\overline{\mM}^{\rm n} , n^*(m \Lambda_{\rm CM}) )$. In particular, $\mM$ is quasi-projective.
\end{thm}

In some sense, the quasi-projectivity of $\mM$ in Theorem \ref{CMemb} could be seen as a consequence of Tian's partial $C^0$-estimates recently established in the fundamental works of Donaldson-Sun \cite{DS2012} and Tian \cite{Tian13}.
Actually such kind of implication was stated without proof in \cite[end of Section 8]{Tian1997} which used the notion of CM stability (introduced in \cite{Tian1997}).  Since we now know that CM stability is equivalent to the existence of K\"{a}hler-Einstein metrics on Fano manifolds (see \cite{Tian14}, \cite{Pa2012}), Viewheg's approach to quasi-projectivity could be applied to study this problem. However, because of the subtlety pointed out by Koll{\'a}r \cite{Kollar06}, it is still not clear to us how to deduce the quasi-projectivity directly using CM stability.  On the other hand, if we only look at the open locus of $\mM$ which parametrizes K\"ahler-Einstein Fano manifolds with finite automorphism groups, then \cite{Don01} and \cite{Odaka14a} have already shown that it is quasi-projective as we know the Fano manifolds it parametrizes are all asymptotically Chow stable by \cite{Don01}. However, we know that if we drop the finite automorphism assumption, there exists K\"ahler-Einstein Fano manifold which is asymptotically Chow unstable (see \cite{OSY}).

Our proof of Theorem \ref{thmext}, which still depends heavily on the recent development in the theory of K\"{a}hler-Einstein metrics on Fano varieties, is also inspired by the work of Schumacher-Tsuji \cite{ST04} and Schumacher \cite{Sch12}. In \cite{Sch12}, Schumacher re-proved the quasi-projectivity of $\mM^{-}$ which is the moduli space of canonically polarized manifolds by using some compactification of $\mM^{-}$ and extension of Weil-Petersson metric. 
Our argument uses a similar approach. 
First, by applying the theory of Deligne pairings,  for any smooth variey $S$ together with a flat family of K\"ahler-Einstein Fano varieties $\cX\to S$ containing an open dense $S^{\circ}\subset S$ such that the fibers of $\cX|_{S^\circ}\to S^\circ$ are all K\"ahler-Einstein Fano manifolds, we can construct a Hermitian metric $h_{\rm DP}$ on the CM line bundle $\lambda_{\rm CM}\to S$ whose restriction to $S^{\circ}$ is the classical Weil-Petersson metric.
Second, the partial-$C^{0}$ estimate established in \cite{DS2012, Tian13} together with an extension of  continuity results in \cite{Li13} allow us to show that this metric is indeed {\em continuous} whose curvature form can be extended to a positive current on $S$.
Third,  by using the local GIT description on the canonical compactification $\overline\cM$ constructed in \cite{LWX}, we can  descend the CM line bundle $\lambda_{\rm CM}$ and the metric $h_{\rm DP}$ to construct an Hermitian line bundle $(\Lambda_{\rm CM}, h_{\rm DP})$ on $\overline{\mM}$, and the curvature form of $(\Lambda_{\rm CM}, h_{\rm DP})$ is the Weil-Petersson current we want. The descending construction is partly inspired by section \cite[Section 11]{FS90} and  based on Kempf's descent lemma proved in \cite{DN89}.
Finally, to obtain the quasi-projectivity, we establish the quasi-projectivity criterion Theorem \ref{NM} for \Red{\em normal} algebraic spaces, which 
can be regarded as \Red{an algebro-geometric version}  of the analytic criterion  in \cite[Section 6]{ST04}.


The paper is organized in the following way: In the next section, we derive some estimate which will play the key role in proving the extension of Weil-Petersson metric.
 In Section \ref{s-can}, assuming the existence of a universal family over a parameter space we  obtain a canonical {\em continuous} Hermitian metric on the CM line bundle with curvature form being a positive current over the base via  the formalism of Deligne pairing . In Section \ref{s-descend}, we descend the metrized CM line bundle to $\overline\cM$ and prove Theorem \ref{thmext} based on a crucial uniform convergence lemma established in Section \ref{append}. In Section \ref{s-proof}, we finish the proof of Theorem \ref{CMemb} by applying the quasi-projective criterion of Theorem \ref{NM}.


\begin{ack}\Red{\Blue{The first author would like to thank Robert Berman for bringing the related question of continuity for Ding energy to his attention. }We would like to thank Gang Tian for useful comments on the history of CM line bundles. We also would like to thank Dan Abramovich, Aise Johan de Jong, Mircea Musta\c{t}\v{a} for helpful discussions.
The first author is partially supported by NSF: DMS-1405936.  The second author is partially supported by a Collaboration Grants for Mathematicians from Simons Foundation.    The third author is partially supported by the grant `The National Science Fund for Distinguished Young Scholars'. }
\end{ack}

\section{Plurisubharmonic functions on complex spaces}\label{secPSH}
For later reference we recall the following known facts. In this section $X$ is a possibly singular complex space. We will denote the open unit disk by $\Delta=\{z\in \mathbb{C}; |z|<1\}$.
\begin{defn}[{\cite[Def. 1, Section 4.1]{GR56}}]\label{defpsh}
A function $\psi(x)$ on $X$ is called plurisubharmonic on $X$, if the following conditions are satisfied:
\begin{enumerate}
\item The value of $\psi(x)$ is real number or $-\infty$.
\item $\psi(x)$ is upper semi-continuous at any point $x_0\in X$: $\displaystyle{\overline{\lim_{x\rightarrow x_0}}} \psi(x)\le \psi(x)$.
\item For any holomorphic map $\tau: \Delta\rightarrow X$, the function $\psi\circ \tau$ is subharmonic on $\Delta$.
\end{enumerate}
\end{defn}
When $X$ is smooth, the above definition recovers the ordinary definition of plurisubharmonic functions on smooth complex manifolds.
\begin{rem}
Note that in the literature, the plurisubharmonic functions in Definition \ref{defpsh} are sometimes called weakly plurisubharmonic functions. The plurisubharmonic functions are then defined as local restrictions of plurisubharmonic functions on $\mathbb{C}^N$ under local embeddings of $X$ into $\mathbb{C}^N$. However, by  a basic result by Fornaess-Narasimhan
\cite[Theorem 5.3.1]{FN80} we know that weakly plurisubharmonic functions are the same as plurisubharmonic functions.
\end{rem}
We have the following important Riemann extension theorem for plurisubharmonic functions (see also \cite[Theorem 5.24]{Dem}):
\begin{thm}[{\cite[Satz 3, Section 1.7]{GR56}}]\label{GRext}
Suppose $D$ is a proper subvariety of $X$ and $\psi^\circ$ is a plurisubharmonic function on $X\setminus D$. Assume that for each point $x\in D$, there exists a neighborhood $U$ such that $\psi$ is bounded from above on $U\setminus(U\cap D)$. Then $\psi^\circ$ extends uniquely to a plurisubharmonic function over $X$.
\end{thm}
This theorem generalizes the following result which is useful for us too:
\begin{thm}[{Brelot, Grauert-Remmert \cite[Satz 5, Section 2.1]{GR56}}]\label{shext}
Assume $\psi^\circ$ is a subharmonic function on $\Delta\setminus\{0\}$ such that $\psi$ is bounded from above in a neighborhood of $0$, then the following function is the unique subharmonic extension of $\psi$ on $\Delta$:
\begin{equation}\label{shexp}
\def\arraystretch{1.2}
\psi(z)=\left\{\begin{array}{rl}
\psi^\circ(z) & \text{ for } z\neq 0,\\
\displaystyle{\overline{\lim_{z\rightarrow 0}}}\; \psi^\circ(z) & \text{ for } z=0.
\end{array}\right.
\end{equation}
\end{thm}
As in \cite[Section 3.3]{Gra62}, it's natural to make the following definition.
\begin{defn}\label{defcur}
A closed positive (1,1)-current $\omega$ on $X$ is by definition a closed positive (1,1)-current $\omega$ on $X^{\reg}$ such that for any $x\in X$, there exists an open neighborhood $U$ of $x\in X$ and a plurisubharmonic function $\psi$ on (the complex space) $U$ such that $\omega|_{U\cap X^{\reg}}=\sqrt{-1}\partial\bar{\partial}\left(\psi|_{U\cap X^{\reg}}\right)$.
\end{defn}

To globalize the above definitions and results, we consider a Hermitian line bundle $(L, h)$ over $X$.  We fix an open covering of $\{U_\alpha\}$ of $X$ and choose generator $l_\alpha$ of $\mathcal{O}_X(U_\alpha)$. Then the Hermitian metric $h$ is represented by a family of real valued functions $\{\psi_\alpha\}$ with $\|l_\alpha\|_h^2=e^{-\psi_\alpha}$. We say that $h$ is continuous if $\psi_\alpha$ is continuous for every $\alpha$. We say $h$ is smooth if $\psi_\alpha=\Psi_\alpha|_{U_\alpha}$ for some local embedding $U_\alpha\rightarrow \mathcal{U}_\alpha\subset\mathbb{C}^N$ and a smooth function $\Psi_\alpha$ on $\mathcal{U}_\alpha$. We say that $(L, h)$ has a positive curvature current if $\Psi_\alpha$ is plurisubharmonic on $U_\alpha$ for every $\alpha$. In this case, we define the Chern curvature current of $(L, h)$ by
\[
c_1(L, h):=\left.\pdJd \psi_\alpha\right|_{X^{\reg}}.
\]
It is easy to verify this is a well-defined positive current in the sense of Definition \ref{defcur}.

\section{Consequence of partial $C^0$-estimate}\label{bypartialC0}
Let $\pi: \cX\rightarrow S$ be a family of smoothable K-polystable Fano variety over a complex space $S$. The consideration here is local in $S$ and so we will assume $S$ is an affine variety in this section. For any $t\in S$, denote by $\cX_t=\pi^{-1}\{t\}$ the fibre above $t$ and by $K_{\cX_t}$ its canonical {$\mathbb{Q}$}-line bundle. Let $(\omega_t, h_t):=(\omega_{\rm KE}(t), h_{\rm KE}(t))$ be the K\"{a}hler-Einstein metric on $(\cX_t, K_{\cX_t}^{-1})$. For any integer $m>0$, we choose $\{\widetilde{s}_i\}_{i=1}^{N_m}$ to be a fixed basis of \Blue{the locally free $\sO_S$-module
$\pi_* \sO_{\cX} (-mK_{\cX})$}, and denote $\widetilde{s}_i(t)=\widetilde{s}_i|_{\cX_t}$. By \cite{DS2012} and \cite{Tian13} (see also \cite[Lemma 8.3]{LWX}), there exists $m_0=m_0(n)>0$, such that for any $m\ge m_0$, we can embed $\widetilde{\iota}_t: \cX_t\hookrightarrow\mathbb{P}^{N_m-1}$ using $\{\widetilde{s}_i(t)\}$ such that $\iota_t^*H_i={\widetilde{s_i}}$ where $H_i$'s { are } coordinate hyperplane sections of $\mathbb{P}^{N_m-1}$. We will fix this identification of $\mathbb{P}\left(H^0(\cX_t, K_{\cX_t}^{-m})^*\right)$ and $\mathbb{P}^{N_m-1}$ from now on. We denote the pull back the Fubini-Study metric by:
\begin{equation}\label{refFS}
\widetilde{\omega}_t:=\frac{1}{m}\widetilde{\iota}_t^*\omega_{\rm FS}=\frac{1}{m}\pdJd\log\sum_{i=1}^{N_m}|\widetilde{s}_i(t)|^2.
\end{equation}
Here the right hand side means that if we choose $e\in \mathcal{O}(K_{\cX/S}^{-m})$ to be any local generator and denote $e_t=e|_{\cX_t}$, then the following is well defined:
\[
\frac{1}{m}\pdJd\log\sum_{i=1}^{N_m}|\widetilde{s}_i(t)|^2=-\frac{1}{m}\pdJd \log \frac{|e_t|^2}{\sum |\widetilde{s}_i(t)|^2}.
\]
\Blue{
We now recall the definition of Bergman kernels and Bergman metrics of $(\cX_t, \omega_{t})$. The metrics $(\omega_t, h_t)$ induce an $L^2$-inner product on $H^0(\cX_t, K_{\cX_t}^{-m})$ as follows:
\[
\la s, s'\ra_{L^2}=\int_{\cX_t} \langle s, s'\rangle_{h^{\otimes m}_t} \omega_t^n.
\]
We choose an orthonormal basis $\{s_i(t)\}$ of $(H^0(\cX_t, K_{\cX_t}^{-m}), \la\cdot,\cdot\ra_{L^2})$ and define the $m$-th {\em Bergman kernel} as follows:
\begin{equation}\label{Bergker}
\rho_m(t):=\rho_{\rm KE}(m,t)=\sum_{i=1}^{N_m}\left|s_i(t)\right|_{h^{\rm \otimes m}_t}^2.
\end{equation}
It is independent of the choice of the orthonormal basis. For the basis  $\{\widetilde{s}_i\}_{i=1}^{N_m}$ we fixed at the beginning,  let  $A_{ij}(t):=\la \widetilde{s}_i(t), \widetilde{s}_j(t)\ra_{L^2}$. Then $\{s_i(t)\}:=A^{-1/2}\{\widetilde{s}_i(t)\}$ is an orthonormal basis.  Now we write
\[
\iota_{t}:=A(t)^{-1/2} \circ \widetilde{\iota}_{t}
\]
with  $\ti\iota_t: \cX_t\hookrightarrow\mathbb{P}^{N_m-1}$ being the embedding given by $\{\widetilde{s}_i(t)\}$,
}
and define the {\em Bergman metric} $\check{\omega}_t$ as following:
\begin{equation}\label{Bergman}
\check{\omega}_t:=\frac{1}{m}\iota_t^*\omega_{\rm FS}=\frac{1}{m}\pdJd\log\sum_{i=1}^{N_m}|s_i(t)|^2=-\frac{1}{m}\pdJd \log \frac{|e_t|^2}{\sum|s_i(t)|^2}.
\end{equation}
{Then by \eqref{Bergker}} and  \eqref{Bergman},  we see that $\omega_t$ and $\check{\omega}_t$ are related to each other via:
\[
\check{\omega}_t=\frac{1}{m}\pdJd\log\rho_m(t)-\pdJd\log|e|_{h_t}^2=\omega_t+\pdJd\left(\frac{1}{m}\log\rho_m(t)\right).
\]
In particular, the K\"{a}hler-Einstein metrics $\omega_t$ satisfies the following complex Monge-Amp\`{e}re equation on {$\cX_t$}:
\begin{equation}\label{KEeq}
\omega_t^n={\rm V} \frac{\rho_m^{1/m}\Omega_{\cX_t}}{\int_{\cX_t}\rho_m^{1/m}\Omega_{\cX_t}}:=e^{-\mathfrak{u}}\Omega_{\cX_t},
\end{equation}
{
where the volume form $\Omega_{\cX_t}$ and the potential $\ku$ in \eqref{KEeq} are given by:
\[
\Omega_{\cX_t}=\left(\sum_{i=1}^{N_m} |s_i(t)|^2\right)^{-1/m}, \quad \mathfrak{u}=-\log\left(\frac{\rho_m^{1/m}}{\int_{\cX_t}\rho_m^{1/m}\Omega_{\cX_t}}\right)-\log {\rm V}.
\]
The right-hand-side of \eqref{KEeq} has that form because we have that $\int_{\cX_t}\omega_t^n={\rm V}:={(c_1(x))^n}$ is a fixed constant. }
{For the purpose of later estimates, we rewrite the right hand side of \eqref{KEeq} into a form using the data from the original (holomorphic) data $\{\widetilde{s}_i\}$:}
$e^{-\ku}\Omega_{\cX_t}=e^{-\tku}\widetilde{\Omega}_{\cX_t}$, such that
\begin{equation}\label{KEeq2}
\tilde{\omega}_t^n= e^{-\tku}\widetilde{\Omega}_{\cX_t},
\end{equation}
where we have denoted:
\[
\widetilde{\Omega}_{\cX_t}=\left(\sum_{i=1}^{N_m} |\widetilde{s}_i(t)|^2\right)^{-1/m}, \quad \tku=\ku-\log\frac{\Omega_{\cX_t}}{\widetilde{\Omega}_{\cX_t}}=-\log\left(\frac{\rho_m^{1/m}\frac{\Omega_{\cX_t}}{\tO_{\cX_t}}}{\int_{\cX_t}\rho_m^{1/m}\frac{\Omega_{\cX_t}}{\tO_{\cX_t}}\tO_{\cX_t}}\right)-\log {\rm V}.
\]
Now denote by $\widetilde{\Omega}=\{\widetilde{\Omega}_{\cX_t}\}$ (resp. $\Omega=\{\Omega_{\cX_t}\}$) the family
of volume forms on $\cX_t$. Then $\tO$ (resp. $\Omega$) defines a Hermitian metric on the relative anti-canonical line bundle $K_{\cX/S}$.  We denote its Chern curvature on the total space $\cX$ by:
\begin{equation}\label{2ref}
-\pdJd \log \widetilde{\Omega}=\widetilde{\omega}\quad (\text{ resp. } -\pdJd\log\Omega=\check{\omega}),
\end{equation}
Then clearly we have: $\widetilde{\omega}|_{\cX_t}=\widetilde{\omega}_t$ (resp. $\check{\omega}|_{\cX_t}=\check{\omega}_t$).

{By \cite{EGZ11}, we know that for a fixed $t\in S$, the function $\ku$ and $\tku$ are continuous on $\cX_t$.} The main result in this section is the following proposition. 
\begin{prop}\label{bdU}
$\ku$ and $\tku$ are continuous and uniformly bounded with respect to $t$.
\end{prop}
\begin{proof}
By using Moser iteration, we know that $\log\rho_m$ is uniformly bounded form above by \cite[(5.2)]{Tian1990}.
Tian's partial $C^0$-estimate states that $\log\rho_m$ is uniformly bounded from below. The partial $C^0$-estimate is {now known to be} true by the fundamental works of \cite{DS2012} and \cite{Tian13}. Moreover, by their proofs, we know that
if $(\cX_{t_i}, \omega_{t_i})$ Gromov-Hausdorff converges to $(\cX_0, \omega_0)$, then in the ambient space ${\mathbb{P}^{N_m-1}}$, we have $\iota_{t_i}(\cX_{t_i})\rightarrow \iota_0(\cX_0)$
and $\rho_m\circ \iota_{t_i}\rightarrow \rho_m\circ\iota_{0}$ uniformly. Now in \cite{LWX} (see Lemma 2.2), we proved that $(\cX_0, \omega_0)$ is indeed {the unique} Gromov-Hausdorff limit as $t\rightarrow 0$ independent of the chosen sequence $\{t_i\}$.

Next we argue that $f(t)=\log\frac{\Omega_{\cX_t}}{\tO_{\cX_t}}$ is continuous. Note that
\[
f(t)=\frac{1}{m}\log\frac{\sum_{i} |\widetilde{s}_i\circ\widetilde{\iota}_t|^2}{\sum_{i}|s_i\circ\iota_t|^2}=\frac{1}{m}\log\frac{\sum_{i}|\widetilde{s}_i(t)|^2}{\sum_{i,j}|A(t)^{-1/2}_{ij}\widetilde{s}_j(t)|^2}.
\]
So we just need to show that $t\mapsto A(t)^{-1/2}$ is a continuous map from $S$ to $GL^{+}(N_m, \mathbb{C})$.
By \cite{DS2012} and \cite{Tian13} we \Blue{can assume} that $A(t)^{-1/2}\{\widetilde{s}_i(t)\}=\{s_i(t)\}\rightarrow \{s_i(0)\}=A(0)^{-1/2} \{\widetilde{s}_i(0)\}$ if $(\cX_{t_i}, \omega_{t_i})\rightarrow (\cX_0, \omega_0)$ in the Gromov-Hausdorff topology. So $A(t_i)^{-1/2}\rightarrow A(0)^{-1/2}$ for such a sequence $\{t_i\}$. Again by Lemma \ref{unique} (see also \cite{LWX}), this holds for any sequence $t_i\rightarrow 0$.

So we have proved the continuity of the function:
\[
F(t)=\frac{1}{m}\log\rho_m+\log\frac{\Omega_{\cX_t}}{\tO_{\cX_t}},
\]
which enters into the expression of $\tku(t)$:
\[
\tku(t)=-\log\left(\frac{e^{F(t)}}{\int_{\cX_t}e^{F(t)}\tO_{\cX_t}}\right)-\log {\rm V}.
\]
The continuity of $\tku(t)$ will then follow by similar arguments as in the proof of \cite[Lemma 1]{Li13} except that we need to use a more general convergence Lemma \ref{kltconv} in the Appendix, which may have independent interest. Indeed, as in \cite[(28)]{Li13} we can estimate:
\begin{eqnarray}\label{decompestimate}
&&\scriptstyle{\left|\int_{\mX_t}e^{F(t)} \tO_{\cX_t}-\int_{\mX_0}e^{F(0)}\tO_{\cX_0}\right|}\\
&\le& \scriptstyle{\left|\int_{\mX_t\backslash\mathcal{W}(\delta)}e^{F(t)}\tOt-\int_{\mX_0\backslash\mathcal{W}(\delta)}e^{F(0)}\tOo\right|+\left|\int_{\mX_t\cap\mathcal{W}(\delta)} e^{F(t)}\tOt-\int_{\mX_0\cap \mathcal{W}(\delta)}e^{F(0)}\tOo\right|}\nonumber\\
&\le& \scriptstyle{\left|\int_{\mX_t\backslash\mathcal{W}(\delta)}e^{F(t)}\tOt-\int_{\mX_0\backslash\mathcal{W}(\delta)}e^{F(0)}\tOo\right|+e^{\|F(t)\|_{L^{\infty}}}\left(\int_{\mX_t\cap\mathcal{W}(\delta)}\tOt+\int_{\mX_0\cap \mathcal{W}(\delta)}\tOo\right)}\nonumber.
\end{eqnarray}
Here $\mathcal{W}(\delta)$ denotes a small neighborhood of $\cX^{\sing}$ in the analytic topology such that $\lim_{\delta \rightarrow 0}\mathcal{W}(\delta)=\cX^{\sing}$ in Hausdorff topology of subsets of ${\mathbb{P}^{N_m-1}\times S}$. By the continuity of $F(t)$ and $\tku_{\cX_t}$ away from $\cX^{\sing}$, the first term on the right hand side of \eqref{decompestimate} can be arbitrarily small if $t$ is sufficiently close to $0$ for a fixed $\delta>0$. On the other hand, by { Lemma \ref{kltconv} }, we have
\begin{equation}\label{lvolconv}
\lim_{t \rightarrow 0} \int_{\cX_t\cap \mathcal{W}(\delta)}\tOt=\int_{\cX_0\cap \mathcal{W}(\delta)}\tOo.
\end{equation}
Now by choosing $\delta$ sufficiently small, we can make the right hand side of \eqref{lvolconv} sufficiently small, and hence the left hand side of \eqref{lvolconv} can also be made sufficiently small
as long as $t$ is sufficiently close to $0$. Combining the above estimates, we indeed see that $\tku(t)$ is continuous as $t\rightarrow 0$.

\end{proof}


\begin{rem}\label{compSch}
{Constrast  to the canonically polarized case studied in \cite{Sch12}, \Red{in which}  the Aubin-Yau's $C^0$-estimate fails to be uniform near  the boundary and hence there is no uniform  lower bound for the K\"ahler  potential of Weil-Petersson metric (see \cite{Sch12, BG13}). Here we have two sided bounds for  the potential of the Weil-Petersson metric for the Fano case.}
\end{rem}

For the reader's convenience, we record the following uniqueness result from \cite{LWX} and {sketch} its proof. \Red{When the automorphism groups are discrete, this is also proved in  \cite{SSY} using a different method}.
\begin{lem}[\cite{LWX}]\label{unique}
In Gromov-Hausdorff topology
$\cX_t\rightarrow \cX_0$, equivalently $\cX_{t_i}\rightarrow {\cX_0}$ for any sequence $t_i\rightarrow 0$.
\end{lem}
\begin{proof}
Assume $\cX_{t_i}\rightarrow \cX_0$ and $\cX_{t'_i}\rightarrow \cX_0'$ such that $\cX_0\neq \cX_0'$. Without loss of generality, we can assume $|t_i|<|t_i'|$ where $|\cdot|$ is any continuous distance function to $0\in S$. Then by the {\em Intermediate Value {Type} result} in \cite[Lemma 6.9.(2)]{LWX}, there exists $t_i''$ such that $|t_i|<|t_i''|<|t_i'|$ such that $(\cX_{t_i''}, \omega_{t_i''})$ converges in Gromov-Hausdorff topology to a K\"{a}hler-Einstein Fano variety $(Y, \omega_Y)$ as $t_i''\rightarrow 0$, which satisfies:
\[
\chow(\cX_{t_i''}, \omega_{t_i''})\rightarrow \chow(Y, \omega_Y) \in \left[\overline{O_{\chow(\cX_0, \omega_0)}}\bigcup\left(GL(N_m)\cdot (\mathcal{U}\cap \overline{O})\right)\right]\setminus O_{\chow(\cX_0, \omega_0)},
\]
where $\chow(\cX_t, \omega_t)$ denotes the Chow point of $\iota_{t}(\cX_t)$ as a subvariety of ${\mathbb{P}^{N_m-1}}$, and
\[
O_{\chow(\cX_0, \omega_0)}=GL(N_m,\mathbb{C})\cdot \chow(\cX_0, \omega_0), \quad \overline{O}=\lim_{t\rightarrow 0} \overline{O_{\chow(\cX_t, \omega_t)}},
\]
and $\mathcal{U}$ is an open set of $\chow(\cX_{0}, \omega_0)$ constructed in \cite[Lemma 3.1]{LWX} using local Luna slice theorem.
If $Y \in GL(N_m,\mathbb{C})\cdot \left(\mathcal{U}\cap \overline{O}\right)$ then by \cite[Lemma 3.1]{LWX} there exists a special test configuration of $Y$ to $\cX_0$. This contradicts to $Y$ being K-polystable (see \cite{Be12}). So we must have $Y\in \overline{O_{\chow(\cX_0, \omega_0)}}\setminus O_{\chow(\cX_0, \omega_0)}$. However, this implies again there is a special test configuration of $\cX_0$ to $Y$ by \cite{Don10}, which contradicts that $\cX_0$ is K-polystable.

\end{proof}


\section{Canonical metric on the CM line bundles}\label{s-can}

In this section, we assume that $\pi: \cX\rightarrow S$ is a family of  K\"{a}hler-Einstein Fano varieties over a {\em smooth complex manifold} (see Remark \ref{normal}) $S$ such that the generic fibre is smooth.  For each $t\in S$, denote by $\omega_t$ the K\"{a}hler-Einstein metrics on $\cX_t=\pi^{-1}(t)$.  Let $A\subset S$ be the analytic set parametrizing
singular $\mathbb{Q}$-Fano varieties (which are smoothable and K-polystable). Denote $S^{\circ}=S\setminus A$ and $\cX^{\circ}=\pi^{-1}(S^{\circ})$. By \cite[Section 5.3]{Don2008} and \cite{Sz2010}, we can assume $\omega_t$ varies smoothly on $\cX^{\circ}$.

\subsection{Preliminaries}
\subsubsection{Weil-Petersson metric and CM line bundle on the smooth locus}\label{WP-CM}
We know that on the open sub-space $S^{\circ}=S\setminus A$ there is a well defined Weil-Petersson metric $\omega_{\rm WP}^{\circ}$. Let's briefly recall its definition and refer to \cite{FS90} for detailed discussions. Fix any $t\in S^{\circ}$. We denote by $\mathcal{T} \cX_t$ the holomorphic tangent sheaf/bundle of $\cX_t$ and by $\mu_t: T_{t}S\rightarrow H^1(\cX_t, \mathcal{T} \cX_t)$ the
Kodaira-Spencer map associated to the family $\cX^{\circ}\rightarrow S^{\circ}$. We denote by $A^{0,k}(X, \mathcal{T} \cX_t) (k\ge 0)$ the space of smooth $\mathcal{T} \cX_t$-valued $(0,k)$-forms on $\cX_t$. Then the K\"{a}hler-Einstein metric $\omega_t$ induces $L^2$ inner products:
\[
(\theta_1, \theta_2)_{L_t^2}=\int_{\cX_t}\la \theta_1, \theta_2\ra_{\omega_t}\omega_t^{n}, \text{ for any } \theta_1, \theta_2\in A^{0,k}(\cX_t, \mathcal{T} \cX_t)
\]
where $\langle\cdot ,\cdot \rangle_{\omega_t}$ is the induced inner product on $\mathcal{T} \cX_t\otimes T^{*(0,k)}\cX_t$.
Then we can define the $L_t^2$-adjoint $\bar{\partial}^*$ of the operator $\bar{\partial}: A^{0,1}(X, \mathcal{T} \cX_t)\rightarrow A^{0,2}(X, \mathcal{T} \cX_t)$ and then the Laplacian operator ${\Box}_{t}=\bar{\partial}^*\bar{\partial}+\bar{\partial}\bar{\partial}^*$ on $A^{0,1}(\cX_t, \mathcal{T} \cX_t)$.  For each $[\theta]\in H^1(\cX_t, \mathcal{T} \cX_t)$ we denote by $\theta_{\rm H}$  the unique harmonic
representative of $[\theta]$. For any $v, v'\in T_t S^{\circ}$, we then define the Weil-Petersson metric by the following formula:
\[
\omega_{\rm WP}^{\circ}(v, v')=(\mu_t(v)_{\rm H}, \mu_t(v')_{\rm H})_{L_t^2}.
\]
We say that a tangent
direction $v\in T_tS$ is {\em  effective} if $\mu_t(v)\neq 0$. Then by its definition, $\omega_{\rm WP}$ is a positive smooth form (i.e. a possibly degenerate metric) over $S^{\circ}$ and is positive definite along directions of effective infinitesimal deformations.
We have the following formula (see also Section \ref{secproj}):
\begin{thm}[{\cite[Theorem 7.9]{FS90}}]
 $\omega_{\rm WP}^{\circ}$ has the following  representation using the fibre integral:
\begin{equation}\label{smWP}
\omega_{\rm WP}^{\circ}=-\int_{\mX^{\circ}/S^{\circ}}\omega_{\mX^{\circ}}^{n+1}.
\end{equation}
Here we have denoted
\begin{equation}
\omega_{\mX^{\circ}}=-\pdJd\log \{\omega_t^{n}\}
\end{equation}
where $\{\omega_t^{n}\}$ is regarded as  an Hermitian metric on $-K_{\mX^{\circ}/S^{\circ}}$.
\end{thm}

As mentioned in the introduction, there is a determinant line bundle equipped with a Hermitian metric whose curvature is equal to the Weil-Petersson metric.
\begin{defn}[\cite{Tian1997}]\label{defCM}
Let $\pi: \cX\rightarrow S$ be a flat family of $\mathbb{Q}$-Fano varieties such that $m K_{\cX/S}$ is Cartier for some integer $m$. We define the $CM$ line bundle $\lambda_{\rm CM}=\lambda_{\rm CM}(S)$ on $S$ as the determinant line bundle associated to the push-forward of a virtual line bundle (in the sense of Grothendieck):
\[
\frac{1}{2^{n+1} m^{n+1}}\det \left[\pi_! \left(-(K_{\cX/S}^{-m}-K_{\cX/S}^{m})^{n+1}\right)\right].
\]
\end{defn}
{
\begin{rem}
Equivalently, we can define the CM-line bundle using Knudsen-Mumford expansion (see \cite{PRS08}, \Red{\cite{PT09}}):
\[
\det\left(\pi_*\left(K^{-m r}_{\cX/S}\right)\right)=-\lambda_{\CM}\frac{(mr)^{n+1}}{(n+1)!}+O(r^{n}).
\]
\end{rem}
}
By Grothendieck-Riemann-Roch theorem, the first Chern class of $\lambda_{\rm CM}(S)$ is given by the formula:
\begin{eqnarray}\label{GRRcurv}
c_1(\lambda_{\rm CM})&=&\frac{1}{2^{n+1}m^{n+1}}\pi_*\left[Ch\left(-(K^{-m}_{\cX/S}-K^{m}_{\cX/S})^{n+1}\right)Td(\cX/S) \right]_{(2)}\nonumber\\
&=&\pi_*\left(-c_1(K_{\cX/S}^{-1})^{n+1}\right).
\end{eqnarray}
\begin{thm}[{\cite[Section 10]{FS90}}]\label{FSQM}
There is a Quillen metric $h_{\rm QM}^{\circ}$ on $\lambda_{CM}|_{S^{\circ}}$ such that
\[
-\frac{\sqrt{-1}}{2\pi}\partial\bar{\partial}\log h_{\rm QM}^\circ=\omega_{\WP}^{\circ}.
\]
\end{thm}

It's natural to expect that h$_{\rm QM}^{\circ}$ in Theorem \ref{FSQM} extends to $h_{\rm QM}$ on $\lambda_{\rm CM}$ over $S$. For this purpose, one needs to study the behavior of the Quillen metric near $A=S\setminus S^{\circ}$ which a priori is difficult (see e.g. \cite{Yos07}). In \cite{Del87} Deligne proposed a program to calculate the Quillen metric (or equivalently the analytic torsion) for general determinant line bundle of cohomology. The (metrized) Deligne pairing in the next sub-section is an example of his approach.

\subsubsection{Deligne pairing with Hermitian metrics}
Let's first recall the definition of Deligne pairings following \cite{Del87, Zha96}. Let $\pi: \cX\rightarrow S$ be a flat and projective morphism of integral schemes of pure relative dimension $n$. For any (n+1)-tuples of
line bundles $\{\mL_{0}, \dots, \mL_{n}\}$ on $\cX$, Deligne \cite{Del87} defined a line bundle on $S$, which is denoted by $\la \mL_0, \dots, \mL_n\ra$ or $\la \mL_0, \dots, \mL_n\ra(\cX/S)$. If $S$ is just one point and $\cX=X$, then
$\la \mL_0, \dots, \mL_n\ra$ is a one-dimensional complex vector space generated by the symbol $\la l_0, \dots, l_n\ra$ (also denote as $\la l_0, \dots, l_n\ra(X)$) where $l_i$ are meromorphic sections whose divisors have empty intersections, with the following relations satisfied. For some $0\le i\le N$ and a meromorphic function $f$ on $\cX$, if the intersection $\bigcap_{j\neq i} {\rdiv}(l_j)=\sum_{\alpha} n_\alpha P_\alpha$ is a 0-cycle and has empty intersection with $\rdiv (f)$, then
\[
\la l_0, \dots, fl_i, \dots, \dots, l_n\ra=\prod_{\alpha} f(P_\alpha)^{n_\alpha} \cdot \la l_0, \dots, l_n\ra.
\]
Now assume each $\cL_i$ has a smooth Hermitian metric $h_i$. Then one can define a metric on $\la \cL_0, \dots, \cL_n\ra$ as follows. For each $0\le i\le n$, let $c_1'(\mL_i)=\frac{1}{2\pi\sqrt{-1}}\partial\bar{\partial}\log h_i$ denote the Chern curvature of $(\cL_i, h_i)$. Then we define inductively (see \cite[8.3.2]{Del87}):
\begin{equation}\label{dp-metric}
\log\| \la l_0, \dots, l_n\ra(X)\|^2=\log\| \la l_0, \dots, l_{n-1}\ra (\rdiv\; l_n)\|^2+\int_{X}\log\|l_n\|^2 \bigwedge_{i=0}^{n-1} c_1'(\mL_j).
\end{equation}


The above construction can then be generalized to the case of a flat family. In that case, the local generator of $\la \mL_0, \dots, \mL_{n}\ra(\cX/S)$ over any Zariski open set $U$ of $S$ are symbols of the form
$\la l_0, \dots, l_n\ra$ where $l_i's$ are meromorphic sections of $\mL_i$ over $\pi^{-1}(U)$ such that
$\bigcap_{i=0}^{n} \rdiv (l_i)=\emptyset$. \eqref{dp-metric} becomes the following induction formula for metrized Deligne pairing (\cite[(1.2.1)]{Zha96}):
\begin{equation}\label{indDP}
\la \mL_0, \dots, \mL_n\ra(\cX/S)=\la \mL_0, \dots, \mL_{n-1}\ra (\rdiv(l_n)/S)\otimes \mathcal{O}\left(-\int_{\cX/S}\log\|l_n\|^2 \bigwedge_{i=0}^{n-1} c_1'(\mL_i)\right) {,}
\end{equation}
where we assume each component of $\rdiv(l_n)$ is flat over $S$ (which can {be} achieved by choosing $l_i$ to be general sections), and $\mathcal{O}(\phi)$ denotes the trivial line bundle over $S$ with metric $\|1\|^2=\exp(-\phi)$. We will need to following regularity result.
\begin{thm}[\cite{Zha96}, \cite{Mor99}]\label{Zhang}
Suppose $h_i$ are smooth metrics on $\cL_i (1\le i\le n+1)$. Then Deligne's metric is continuous on $\la \cL_1, \dots, \cL_{n+1}\ra$.
\end{thm}
Deligne's metric is important for us because its curvature is given by the appropriate fiber integral.
\begin{thm}[{\cite[Proposition 8.5]{Del87}}]
The following curvature formula holds for Deligne's metric:
\begin{equation}\label{DPcurv}
c_1'\left(\la \cL_0, \dots, \cL_n\ra\right)=\int_{\cX/S}c_1'(\cL_0)\wedge \dots \wedge c_1'(\cL_n).
\end{equation}
\end{thm}
{
\begin{rem}
Notice that the right-hand-side is well defined. See for example \cite[Section 3.4]{Var89}. Actually, the regularity result in Theorem \ref{Zhang} is closely related to the results in \cite{Var89}. See Proposition 3.4.1, Theorem 2 and Theorem 3 in \cite{Var89}.
\end{rem}
}

Using the inductive formula in \eqref{indDP}, we immediately get:
\[
\la \mL_0, \dots, \mL_n\otimes\mathcal{O}(\phi)\ra = \la \mL_0, \dots, \mL_n\ra \otimes \mathcal{O}\left(\int_{\cX/S}\phi \bigwedge_{i=0}^{n-1} c_1'(\mL_i)\right).
\]
Since the Deligne pairing is symmetric, we get the following {\em change of metric formula} (see \cite[(2.8)]{PS04}, \cite{Tian2000}):
\[
\la \mL_0\otimes \mathcal{O}(\phi_0), \dots, \mL_n\otimes \mathcal{O}(\phi_k)\ra =\la \mL_0, \dots, \mL_n\ra \otimes \mathcal{O}\left(\int_{\cX/S}\sum_{j=0}^n \phi_j \bigwedge_{k<j}c_1'(\cL_k\otimes\mathcal{O}(\phi_k))\bigwedge_{l>j}c_1'(\mL_l)\right).
\]
In particular, if $\mL_i=\mL$ and $\phi_i=\phi$ are the same, then we have:
\begin{equation}\label{comform}
\left(\mL\otimes \mathcal{O}(\phi)\right)^{\la n+1\ra}=\mL^{\la n+1\ra}\otimes \mathcal{O}\left(\sum_{j=0}^{n}\int_{\cX/S} \phi\; c_1'(\mL\otimes \mathcal{O}(\phi))^{n-j}\wedge c_1'(\mL)^j\right).
\end{equation}
Here we have denoted:
\[
\mL^{\la n+1\ra}=\langle \overbrace{\mL, \dots, \mL}^{(n+1) \text{ times}}\rangle.
\]
\Blue{
The following simple but useful lemma follows immediately from the functorial construction of the metrized Deligne pairings.
\begin{lem}\label{isom}
For any base change $g: S'\rightarrow S$, denote the pull-back family by $(g^*\cX, g^*\cL):=(\cX,\cL)\times_{g,S}S'$ where $g^*\cL$ is endowed with the pull-back metric.
Then we have an isometric isomorphism: $\left(g^*\cL\right)^{\la n+1\ra}\cong g^*\left(\cL^{\la n+1\ra}\right)$.
In particular, if the group $G$ acts equivariantly on $(\cX/S, \mL)$, then for all $\sigma\in G$, we have an isometry:
$
\sigma^*\mL^{\la n+1\ra}\cong (\sigma^*\mL)^{\la n+1\ra}.
$
\end{lem}
}
%
\subsection{Canonical continuous hermitian metric on the Deligne pairing}\label{DP}
We can apply the Deligne pairing to study the CM line bundle, because it's known that:
\begin{thm}[\cite{Del87, PRS08, PT09, Be12}]
\[
\lambda_{\rm CM}=-\left(K_{\cX/S}^{-1}\right)^{\la n+1 \ra}.
\]
\end{thm}
Notice that by \eqref{GRRcurv} and \eqref{DPcurv} the curvatures on both sides are the same.
In this section, we will use the formalism of Deligne pairings to prove the following main result:
\begin{thm}\label{exthDP}
There exists a continuous metric $h_{\rm DP}$ on $\lambda_{\CM}$ over $S$ such that 
\begin{enumerate}
\item
$\omega_{\WP}:=- \frac{\sqrt{-1}}{2\pi}\partial\bar{\partial}\log h_{\rm DP}$ is a positive current.
\item $\omega_{\rm WP}|_{S^\circ}=\omega_{\WP}^\circ.$
\end{enumerate}
\end{thm}
\begin{proof}[Proof of Theorem \ref{exthDP}]
Recall that $\tO=\{\tO_{\cX_t}\}$ defines a smooth metric on $K_{\cX/S}^{-1}$ which is nothing but the pull back of the Fubini-Study metric {by} a holomorphic family of embeddings (see Section \ref{bypartialC0}). By Theorem \ref{Zhang}, the associated metric $\widetilde{h}_{\rm DP}$ on $\lambda_{\rm CM}=-\left(K_{\cX/S}^{-1}\right)^{\la n+1\ra}$ is continuous.
From now on we choose an open covering $\{\mathcal{U}_\alpha\}$ of $S$ and generators $l_\alpha$ of $\mathcal{O}_{S}(\lambda_{\CM})(\mathcal{U}_\alpha)$, {such that}
\[
\|l_\alpha\|_{\widetilde{h}_{\rm DP}}^2=e^{-\widetilde{\Psi}_\alpha}{,}
\]
where $\widetilde{\Psi}_\alpha$ is a continuous function on $\mathcal{U}_\alpha$.
Moreover, by curvature formula in \eqref{DPcurv} we have:
\begin{equation}\label{KrefS}
\pdJd\widetilde{\Psi}_\alpha=-\int_{\cX/S} \widetilde{\omega}^{n+1}=:\widetilde{\omega}_S.
\end{equation}
Notice that $\widetilde{\omega}_S$ is however not known to be a positive K\"{a}hler form. Next we bring in the K\"{a}hler-Einstein metrics $\{\omega_t\}$. By \eqref{KEeq2} we have:
\begin{equation}\label{KEeq2b}
\omega_t^n=e^{-\tku}\widetilde{\Omega}_{\cX_t}.
\end{equation}
$\{\omega_t^n\}$ (resp. $\tO$) defines a continuous (resp. smooth) metric on $K_{\cX/S}^{-1}$. So by the {\em change of metric formula} in \eqref{comform}, we define:
\begin{equation}\label{KEhDP}
h_{\rm DP}=\widetilde{h}_{\rm DP}\cdot e^{-\kU}
\end{equation}
where
\begin{equation}\label{Psi}
\kU=-\sum_{j=1}^n\int_{\cX/S}\tku\;\widetilde{\omega}_t^j\wedge \omega_t^{n-j}.
\end{equation}
Over each fibre $\cX_t$, $\widetilde{\omega}_t$ is a smooth K\"{a}hler metric (pull back of Fubini-Study) and $\omega_t=\widetilde{\omega}_t+\pdJd\tku$ is positive current with continuous potentials. So by pluripotential theory $\kU$ is a well defined function on $S$.   Moreover, by Lemma \ref{bdU}, we know that $\tku$ is continuous and uniformly bounded on $\cX$. Then we can show that $\kU$ in \eqref{Psi} is continuous and uniformly bounded by the following Lemma \ref{bdPsi} (see \cite{Li13, SSY}).  So we conclude that $h_{\rm DP}$ is a continuous Hermitian metric
on $\lambda_{\CM}$.

Next we look at these data at the curvature level. By taking the curvature on both sides of \eqref{KEeq2b}, we have:
\begin{equation}\label{otou}
\omega_{\cX}=\pdJd \log\{\omega_t^n\}=\widetilde{\omega}+\pdJd \tku.
\end{equation}
From \eqref{otou} we see that on $\cX^{\circ}$ the following identity holds:
\begin{eqnarray*}
\omega_{\cX^{\circ}}^{n+1}&=& \widetilde{\omega}^{n+1}+\pdJd \left(\tku\sum_{j=0}^{n} \widetilde{\omega}^{j}\wedge \left(\widetilde{\omega}+\pdJd \tku\right)^{n-j}\right),
\end{eqnarray*}
Since over $\cX^{\circ}$ both $\omega_{\cX^{\circ}}$ and $\widetilde{\omega}$ are smooth (1,1)-forms and $\tku$ is a smooth function, we can do { fibre integrals} to get:
\begin{equation}\label{decomp}
-\int_{\cX^\circ/S^\circ} \omega_{\cX^\circ}^{n+1}=
-\int_{\cX^{\circ}/S^{\circ}}\widetilde{\omega}^{n+1}+\pdJd \kU,
\end{equation}
where $\kU$ was defined in \eqref{Psi}. By \eqref{smWP} and \eqref{KrefS}, we see that the above identity is equivalent to:
\[
\omega_{\rm WP}^\circ=-\pdJd \log \widetilde{h}_{\rm DP}+\pdJd\kU={-\left.\pdJd\log h_{\rm DP}\right|_{S^{\circ}}}.
\]
Note that $\omega_{\rm WP}^{\circ}$ is a {positive form} over $S^\circ$. Now the conclusion of the Theorem follows from Lemma \ref{context}.
\end{proof}
\begin{rem}
One may would like to try to extend $\omega_{\rm WP}^\circ$ to $S$ by defining
\begin{equation}\label{trydef}
\omega_{\rm WP}=-\int_{\cX/S} \omega_{\cX}^{n+1}
\end{equation}
The problem is that we need to verify that this $\omega_{\rm WP}$ in \eqref{trydef} is well-defined. There are several technical difficulties. For one thing, $\omega_{\cX}$ can not be
a positive current because we expect $\omega_{\rm WP}$ to be positive. So this would prevent us to define the ``Monge-Amp\`{e}re measure" $\omega_{\cX}^{n+1}$. Even we
could define this wedge product, we still need to make sense of the fibre integral, for which the work in \cite{Var89} may be helpful. To get around these difficulties, we proved the above result by combining the continuity and the uniqueness of plurisubharmonic extension. 
\end{rem}

\begin{lem}\label{bdPsi}
The function $\kU$ is continuous and uniformly bounded on $S$.
\end{lem}

\begin{proof}
The argument to prove this was known by \cite{Li13} (see also \cite{Be12} and \cite{SSY}). For the reader's convenience, we briefly sketch the proof and refer to \cite{Li13} and \cite{SSY} for  more details.
We can estimate in the way similar to \eqref{decompestimate},
\begin{eqnarray*}
\left|\int_{\mX_t} \tku \widetilde{\omega}_t^j\wedge \omega_t^{n-j}-\int_{\mX_0}\tku \widetilde{\omega}_0^j\wedge \omega_0^{n-j}\right|&\le& \left|\int_{\mX_t\backslash\mathcal{W}(\delta)}\tku \widetilde{\omega}_t^j\wedge\omega_t^{n-j}
-\int_{\mX_0\backslash\mathcal{W}(\delta)}\tku \widetilde{\omega}_0^j\wedge\omega_0^{n-j}\right|+ \\
&&\|\tku\|_{L^\infty} \left(\int_{\mX_t\cap\mathcal{W}(\delta)}\widetilde{\omega}_t^j\wedge \omega_t^{n-j}
+\int_{\mX_0\cap \mathcal{W}(\delta)}\widetilde{\omega}_0^j\wedge \omega_0^{n-j}\right)\nonumber.
\end{eqnarray*}
As in \eqref{decompestimate}, $\mathcal{W}(\delta)$ is a sufficiently small neighborhood of $\cX^{\sing}$.
$\|\tku\|_{L^\infty}$ is finite because of Proposition \ref{bdU}.

Now to estimate the first term, we can choose a partition of unity $\{\mathcal{V}_\alpha, \rho_\alpha\}$ of $\cX\setminus \mathcal{W}(\delta)$ such that $\pi: \mathcal{V}_\alpha\rightarrow S$ is a local
fibration. Note that both $\widetilde{\omega}_t$ and $\omega_t=\widetilde{\omega}_t+\pdJd\tku|_{\cX_t}$ are positive currents with {\em continuous} potentials, and $\tku$ is continuous by Lemma \ref{bdU}. So by convergence of Monge-Amp\`{e}re measures (see \cite[Corollary 3.6, Chapter 3]{Dem}), we can show that
\begin{equation}\label{conv1}
\lim_{t\rightarrow 0}\int_{\mV_\alpha\cap \cX_t}\rho_\alpha \tku\; \widetilde{\omega}_t^j\wedge \omega_t^{n-j}=\int_{\mV_\alpha\cap \cX_0}\rho_\alpha \tku\; \widetilde{\omega}_0^j\wedge \omega_0^{n-j}.
\end{equation}
By patching together the convergence \eqref{conv1} on all $\mV_\alpha$, we see that the first term approaches $0$ as $t\rightarrow 0$ {for any fixed $\delta$}.
Note that when we write $t\rightarrow 0$, we mean the limit holds for any sequence $t_i\rightarrow 0$.
Because $\cX^{\sing}$ is a pluripolar set, it's immediate that for any $\epsilon>0$, there exists $0<\delta\ll 1$ such that
\begin{equation}\label{conv2}
\int_{{\cX_0\cap \mathcal{W}(\delta)}}\widetilde{\omega}_0^j\wedge\omega_0^{n-j}\le \epsilon
\end{equation}
Lastly to estimate the first term in the bracket, we {can} use the following trick (\Red{see \cite{SSY} and \cite{Li13}}):
\begin{equation}\label{conv3}
\int_{\cX_t\cap \mathcal{W}(\delta)}=\left(\int_{\cX_t}-\int_{\cX_t\setminus\mathcal{W}(\delta)}\right)\Red{.}
\end{equation}
Now $\int_{\cX_t}\widetilde{\omega}_t^j\wedge \omega_t^{n-j}={ K_{\cX_t}^{-n} = K_{\cX_0}^{-n}}=\int_{\cX_0}\widetilde{\omega}_0^j\wedge\omega_0^{n-j}$ is a constant independent of $t$ and
\begin{equation}\label{conv4}
\lim_{t\rightarrow 0}\int_{\cX_t\setminus \mathcal{W}(\delta)}\widetilde{\omega}_t^j\wedge\omega_t^{n-j}=\int_{\cX_0\setminus\mathcal{W}(\delta)}\widetilde{\omega}_0^j\wedge \omega_0^{n-j}
\end{equation}
using the similar reasoning in \eqref{conv1}. Then by combining \eqref{conv1}-\eqref{conv4}, we show that the first term in the bracket can indeed be made arbitrarily small as long as $t$ and $\delta$ are sufficiently
small. So we are done.

\end{proof}

\begin{lem}\label{context}
Let $S$ be a smooth complex manifold, and $(L, h)\rightarrow S$ be a line bundle with a continuous Hermitian metric $h$. Let $A\subset S$ be a proper subvariety. Denote by {$c_1(L, h)$ the curvature current of $h$. If $c_1(L,h)^\circ:=c_1(L,h)|_{S^\circ}$ is a positive current over $S^\circ=S\setminus A$, then $c_1(L,h)$ is a positive current over the whole $S$.}
\end{lem}
\begin{proof}
 Choose a covering $\{\mathcal{U}_\alpha\}$ of $S$.
The metric $h$ is locally represented by a positive function $\exp(-\Psi_\alpha)$ over $\mathcal{U}_\alpha$. From the assumption, we know that $\Psi_\alpha$ is continuous (in particular locally uniformly bounded) on $\mathcal{U}_\alpha$ and plurisubharmonic on $\mathcal{U}_\alpha\setminus S$. By Theorem \ref{GRext} (see also \cite[Theorem 5.23, 5.24, Chapter 1]{Dem}) we know that $\Psi_\alpha^{\circ}:=\Psi_\alpha|_{\mathcal{U}_\alpha\setminus A}$ extends {\em uniquely} across $A$ to become a plurisubharmonic function on $\mathcal{U}_\alpha$. Because $\Psi_\alpha$ is continuous, this extension must coincide with $\Psi_\alpha$ itself, { as can be seen} by restricting to any analytic curve and using the expression \eqref{shexp} in Theorem \ref{shext}.  So we see that $\Psi_\alpha$ is plurisubharmonic on $\mathcal{U}_\alpha$. { Using the Definition \ref{defcur} in Section \ref{secPSH}, we immediately transform this into the statement of the lemma.}
\end{proof}

\begin{prop}\label{GE}
$CM$-line bundle $\lambda_{\rm CM}$ is nef over $S$. Moreover, if $S$ is proper such that $\cX^{\circ}\rightarrow S^{\circ}$ is a {\em generically effective} family, then $\int_{S}c_1(\Lambda_{\rm CM})^{\dim S}>0$.
\end{prop}
\begin{proof}
$\lambda_{\rm CM}\cdot C=\int_{C}\omega_{\rm WP}\ge 0$ for any curve $C\subset S$. Because $\omega_{\rm WP}$ has continuous bounded
potentials, the Monge-Amp\`{e}re measure and hence the integral is well defined. The last statement is true because $\omega_{\rm WP}|_{S^{\circ}}=\omega_{\rm WP}^{\circ}$ is positive definite along effective directions.
\end{proof}

We conclude this section with the following remark.
\begin{rem}\label{normal}
The results in this section still hold true when we just assume $S$ is a normal complex space. To see this, we choose a resolution of singularities $f: \widetilde{S}\rightarrow S$ and consider the induced family of K-polystable Fano varieties $\widetilde{\cX}:=\cX\times_S\widetilde{S}$. Then we can carry out the above constructions for the new family $\widetilde{\cX}\rightarrow \widetilde{S}$. Because CM line bundle is functorial, we have $f^*\lambda_{\CM}(S)=\lambda_{\CM}(\widetilde{S})$. For the data of positively curved metrics, because the fiber of $f$ is connected by Zariski's main theorem, it's easy to see that the data obtained over $\widetilde{S}$ naturally descends to $S$.
\end{rem}

\section{Descending  CM line bundle and Deligne metrics to the Moduli space}\label{s-descend}
In this section, we will construct the descendent of CM line bundle on $\overline\cM$, the proper moduli space constructed in \cite{LWX}. Following \cite[Section 8]{LWX}, let us introduce the following parameter space with one minor difference, that is, we will work with  Hilbert schemes instead of Chow variety.
\begin{defn}\label{Z0}

\begin{equation}\label{Z}
\Red{Z}:=\left \{\hilb(Y)\left | \mbox{\stackbox[v][m]{{\footnotesize
$Y\subset \PP^{N_m-1}$ \Red{is} a smooth Fano manifold with \\
\Red{$\dim H^0(K_Y^{\otimes t})=\chi(t), \forall t\gg 1$} and  $\left.\sO_{\PP^{N_m-1}}(1)\right|_Y\cong K_Y^{-\otimes m}$.
}}}\right .\right \}\subset\hilb(\PP^{N_m-1},\chi)\ .
\end{equation}	

By the boundedness of smooth Fano manifolds with fixed dimension (see \cite{KMM92}), we may choose $m\gg1$ such that  $\chi(m)=N_m$ and $Z$ includes all such Fano manifolds. Now following \cite[Section 8]{LWX}, let $\overline{Z}\subset \hilb(\PP^{N_m-1},\chi)$ be the closure of $Z$ inside $ \hilb(\PP^{N_m-1},\chi)$, $Z^{\kss}$ (resp. $Z^\kps \subset Z^\kss$ ) be the {\em open } (resp.  {\em constructible}) subset of $\overline{Z}$  parametrizing the {\em K-semistable } (resp. {\em K-polystable}) $\mathbb{Q}$-Fano subvarieties, and $(Z^{\kps})^{\circ} (\mbox{ resp. } (Z^{\kss})^{\circ}) \subset Z^{\kps}$  be the subset parametrizing smooth K-polystable {\em Fano manifolds}.	
Let $Z^\ast$ denote  the seminormalization of $Z^\kss_{\rm red}$, the reduction of $Z^\kss$ and  $(Z^{\rm kps})^\ast=Z^\kps\times_{Z^\kss} Z^\ast$  denote   the pull back of $Z^\kps$.
\end{defn}


Let  $\cX\to Z $ be the universal family over $Z$ and by abusing of notation we will still let $\cX \to Z^\ast$ denote the pull back. 
Now for each $z\in Z^{\kps}$, the corresponding $\cX_z$ is equipped with a weak K\"ahler-Einstein metric $\omega_z$. Its volume form $\omega_z^n=e^{-\tku}\tO$ defines a continuous Hermitian metric on $K_{\cX_z}^{-1}$, hence a Hermitian metric $h_{\rm DP}(z)$ on $\lam_\CM|_{z}$ by \Red{\eqref{KEhDP}.} 
Now our main result of this section is the following equivalent version of Theorem \ref{thmext}.
\begin{thm}\label{descend}
Let $\overline\cM$ be the proper good moduli space (cf. \cite[Theorem 1.3]{LWX}) for the quotient stack $[Z^\ast/SL(N+1)]$. Then there is a $k=k(r,\chi)$ such that the bundle $\lam_\CM^{\otimes k}\to Z^\kss$  descends to a  $\QQ$-line bundle $\Lam_\CM$ on $\overline\cM$ with a well defined continuous metric $h_{\rm DP}$, whose curvature is a positive current.
\end{thm}

Before we descend $\lam_\CM\to Z^\ast$ to $\overline\cM$,  let us recall the theory developed \cite{Alp13} and  \cite{AFSV14}.
\begin{defn}\label{pair-pr}
Let $\cZ$ be an algebraic stack of finite type over $\CC$, and let $z\in \cZ(\CC)$ be a closed point with reductive stabilizer $G_z$. We say $f_z:\cV_z\to \cZ$ is a {\em local quotient presentation around $z$} if
\begin{enumerate}
\item $\cV_z=[\spec A/G_z]$, with $A$ being a finite type $\CC$-algebra.
\item $f_z$ is \'etale and affine.
\item There exists a point $v\in \cV_z$ such that $f_z(v)=z$ and $f_z$ induces isomorphism $G_v\cong G_z$.
\end{enumerate}
We say $\cZ$ {\em admits a local quotient  presentation } if there exists a local quotient presentation around every closed point  $z\in\cZ$.
\end{defn}
Then we have the following
\begin{thm}[Theorem 10.3 in \cite{Alp13} and Theorem 4.1 in \cite{AFSV14}]\label{good-line}
Let  $\cZ$ be an algebraic stack  of finite type over $\CC$,
\begin{enumerate}
\item For every closed point $z\in \cZ$, there is a {local quotient presentation $f_z: \cV_z\to \cZ$ around $z$} such that
\begin{enumerate}[a)]
\item \label{Gz-prsv}$f_z$ is stabilizer preserving at  closed points of $\cV_z$, i.e. for any $v\in \cV_z(\CC)$, $\aut_{\cV_z(\CC)}(v)\to \aut_{\cZ(\CC)}(f(v))$ is an isomorphism.
\item $f_z$ sends closed points to closed points.
\end{enumerate}
\item For any $\CC$-point $z\in \cZ$, the closed substack $\overline{\{z\}}$ admits a good moduli spaces.
\end{enumerate}
Then $\cZ$ admits a good moduli space $M$. Furthermore, if $\cZ$ admits a line bundle $\cL$ such that for any closed point $z\in \cZ(\mathbb{C})$, the stabilizer $G_z$ acts on $\cL|_z$ trivially, then $\sL$ descends to a line bundle $L$ on $M$.
\end{thm}

\begin{rem}
The local condition for descending the line bundle to a good quotient already appeared in \cite[Theorem 2.3]{DN89}.
See also \cite[Lemma 11.7]{FS90} and the discussion in \cite[Section 6.2]{OSS}.
\end{rem}
To apply the theorem above,  let us recall the local GIT picture described in \cite[Theorem 8.5]{LWX}. For any isomorphic class $[z]\in {\overline{\cM}}$,  there is a unique $\SL(N_m)$-orbit $O_z\subset Z^\kps$ lying over $[z]$. If we let $\hilb(\PP^{N_m-1},\chi)\subset \PP^K$ be the Pl\"ucker embedding, then for any fixed representative $z\in O_z$ there is a $\Aut(\cX_z)$-invariant linear subspace $z\in\PP W\subset \PP^K$ and a $\Aut(\cX_z)$-invariant open neighborhood $z\in\cU_z\subset \PP W\cap Z^\kss$ such that the GIT quotient \Red{$[\cU_z\sslash\Aut(\cX_z)_0]$} gives rise to a local \'etale chart $[z]\in \cV_{[z]} \subset {\overline{\cM}}$ around $[z]$, where \Red{$\aut(\cX_z)_0$  is the identity component of $\aut(\cX_z)$} . This is precisely the local quotient presentation needed for the quotient stack $[Z^\ast /\SL(N_m)]$. By abusing the notation, we will still let $\cU_z$ to denote $\cU_z\times_{Z^\kss_{\rm red}}Z^\ast$. \Red{Furthermore, in order to descend $\lam_\CM$, we  need the stabilizer $\aut(\cX_z)$ of any {\em closed} point $z\in Z(\CC)$ acts trivially on $\lam_\CM|_z$. By \cite[Theorem 1.3]{LWX},  we have a local quotient presentation $[\cU_z\sslash\aut(\cX_z)_0]$ around $z$ such that  $z\in \cU_z$  is GIT  poly-stable with respect to $\Aut(\cX_z)_0$-linearization of $\sO_{\PP W}(1)|_{\cU_z}$.
On the other hand, by \cite[8.5]{LWX} $\lambda_\CM|_{z}$ is K-polystable since it is GIT polystable. So the Futaki invariant vanishes and hence the Lie algebra $\faut_{z}$  acts trivially on $\lam_\CM|_{z}$. Indeed, it's now well known that the Futaki invariant (\cite{Fut83}, \cite{Fut90}, \cite{DT92}) is the same as the action of $\faut_{z}$ on $\lambda_{\CM}|_{z}$ (see \cite{Tian1997}, \cite{Don00}, \cite{PT09}).  In order to trivialize the action of $\aut(\cX_z)$, let us introduce $k_z:=|\Aut(\cX_z)/\Aut(\cX_z)_0|$.  Then $\aut(\cX_z)$ acts trivially on $\left.\lam_\CM^{\otimes {k_z}}\right|_z$.}

\begin{lem}\label{uni-k}
$k_z$ is uniformly bounded for $z\in Z^\ast$, i.e. $k_z<k=k(m,\chi)$ with $m,\chi$ being fixed in \eqref{Z}.
\end{lem}
\begin{proof}Let us consider the universal family $\cX\to Z^\ast$, which is a bounded family. Then $\aut_{Z^\ast}\cX:={\rm Isom}_{Z^\ast}(\cX,\cX)$ is a group scheme over $Z^\ast$. In particular, this implies that number of component over each Zariski open set of $Z^\ast$ is uniformly bounded.
\end{proof}
As a consequence,  the action of $\aut(\cX_z)$ on $\lam_\CM^{\otimes k}$ is trivial for all closed point $z\in Z(\CC)$.
 This together with the proof of \cite[Theorem 8.5]{LWX} imply all the assumptions of Theorem \ref{good-line} are met, and hence $\lam_\CM^{\otimes k}$ descends to a line bundle $\Lam^{\otimes k}_\CM$ over $\overline\cM$ as we desired.


With $\Lam_\CM$ in hand, we may proceed the proof of the main result of this section.
\begin{proof}[Proof of Theorem \ref{descend} (=Theorem \ref{thmext})]
To finish the proof, we need to descend $h_{\rm DP}$ to a metric on $\Lam_\CM$. To do that, let us fix $[z]\in \overline{\mM}$. We choose $z\in Z^{\kps}$ and a $\aut(\cX_z)$-equivariant slice $\mU_z$ such that $\mU_z\sslash \aut(\cX_z)=\cV_{[z]}$ is an \'etale neighborhood of $[z]$ as in the \cite[Theorem 8.5]{LWX}.  Fix a generator $\mathfrak{l}_z=\la l_0, \dots, l_n\ra(\cX_z)$ of $\lambda_{\rm CM}(\cX_z)$. By the proof of \cite[Theorem 2.3]{DN89} or \cite[Theorem 10.3]{Alp13}, we see that $\mathfrak{l}_z$ can be extended to an $\aut(\cX_z)$-invariant section $\mathfrak{l}\in H^0\left(\cU_z,\left.\lam^{\otimes k}_\CM\right|_{\cU_z}\right)^{\aut(\cX_z)}$, which  descends to a local section $[\mathfrak{l}]\in H^0\left(\cV_{[z]},\left.\Lam_\CM^{\otimes k}\right |_{\cV_{[z]}}\right)$. Then we define a Hermitian metric on $\Lambda_{\rm CM}$ by
\[
\left\|[\mathfrak{l}]\right\|_{h_{\rm DP}}([z])=\|\mathfrak{l}\|_{h_{\rm DP}}(z).
\]
By $\aut(\cX_z)$-equivariance of the metrized Deligne pairing in Lemma \ref{isom}, the Hermitian metric $h_{\rm DP}$ on $\Lambda_{\rm CM}$ is well defined. Now we claim that $h_{\rm DP}$ is continuous on $\Lambda_{\rm CM}$.
For that, let $[z_i]\stackrel{i\to \infty}{\longrightarrow} [z]$ in $\overline{\mM}$ be a sequence and  $z_i\to z\in \mU_z\cap Z^{\kps}$ be the lifting, we need to show that $\|\mathfrak{l}\|_{h_{\rm DP}}(z_i)\rightarrow \|\mathfrak{l}\|_{h_{\rm DP}}(z)$. \Red{Recall that, by the change of metric formula  \eqref{comform}, we have formula \eqref{KEhDP}:}
\[
\|\mathfrak{l}\|_{h_{\rm DP}}^2(z_i)=\|\mathfrak{l}\|_{\widetilde{h}_{\rm DP}}^2 e^{-\kU_{i}}(z_i),
\]
where
\[
\kU_i=-\sum_{j=0}^n \int_{\cX_{z_i}}\tku_{z_i} \omega_{z_i}^j\wedge \widetilde{\omega}_{z_i}^{n-j}
\]
and $\widetilde{h}_{\rm DP}$ is the Deligne metric on $\lam_\CM$ defined using the volume form $\ti \Omega$ on $K^{-1}_{\cX/Z^\ast}$.
By Theorem \ref{Zhang}, we know that $\Red{\|\mathfrak{l}\|^2_{\widetilde{h}_{\rm DP}}(z_i)\rightarrow \|\mathfrak{l}\|^2_{\widetilde{h}_{\rm DP}}(z)}$.
By the proof of Lemma \ref{bdPsi}, all we need is that
\begin{equation}\label{tU}
\tku(z_i)\stackrel{i\to \infty}{\longrightarrow} \tku(z).
\end{equation}
Now by our construction  $\cX|_{\cU_z}\to\cU_z$ is a  family of klt Fano varieties.
So by the proof of Proposition \ref{bdU}, 
\eqref{tU} is a consequence of
\[
\lim_{z'\rightarrow z} \int_{\cX_{z'}}\tO_{\cX_{z'}}=\int_{\cX_z}\tO_{\cX_z}, \text{ for } z'\in \mU_z.
\]
which will be proved in Lemma \ref{kltconv}.

Finally we show that $(\Lambda_{\CM}, h_{\rm DP})$ has positive curvature in the sense of Definition \ref{defcur} in Section \ref{secPSH}.  As explained at the end of Section \ref{secPSH}, this question is local.
By the continuity of $h_{\rm DP}$ and Theorem \ref{GRext}, we just need to verify the positivity over $\mM$.
By Definition \ref{defpsh}, we need to verify the positivity along any analytical curve. So letting $\tau: \Delta\rightarrow \cM$ be any holomorphic map, we need to verify the positivity for $(\tau^*\Lambda_{\CM}, \tau^*h_{\rm DP})$. After possibly finite base change 
\Red{$p_1:\Delta\to \Delta$,  we can lift $\tau$ to a holomorphic map: ${\tau}_1: \Delta\rightarrow (Z^{\kss})^{\circ}$, such that $\tau_1 (\Delta^{\circ})$ is contained in  a component of $(Z^{\kps})^{\circ}$ where $ \Delta^{\circ}= \Delta\setminus \{\mbox{finite points}\}$. However, by \cite[3.1]{LWX}, we know that after shrinking $\Delta$ and replacing $\tau$, we can always assume that  for every point $t\in \Delta$, $\tau_1(t)\in (Z^{\kps})^{\circ}$.
}

\Blue{
Let  $\tilde{\tau}=p_2\circ \tau_1$ with $p_2:  (Z^{\kss})^{\circ}\to \mM$ be the quotient morphism. Since $p_1$ is generically smooth and $h_{\rm DP}$ is continuous, by Theorem \ref{shext} we just need
to verify the positivity of $(\tilde{\tau}^*\lambda_{\CM}, \tilde{\tau}^*h_{\rm DP})$. Now since $\cX\times_{\tilde{\tau}, (Z^{\kss})^{\circ}} \Delta$ is a flat family K\"{a}hler-Einstein Fano manifolds over $\Delta$, we get the positivity by the positivity of $\omega_{\rm WP}^\circ$ explained in Section \ref{WP-CM} (see also the proof of Theorem \ref{CMemb} in the next Section).
}
\end{proof}
\begin{rem}
It may be possible to verify the continuity of $h_{\rm DP}$ using directly the Hermitian metric on $K_{\cX/S}^{-1}$ by $\{\omega_t^n\}$. By \cite{CC97}, we know that the the volume measure is continuous under the GH convergence. So we indeed expect that the Hermitian metric $\{\omega_t^n\}$ changes continuously with respect to $t$ so that the metric on the Deligne pairing changes continuously. However, since the volume measure is not exactly the same as the volume form $\omega_t^n$, some extra arguments are needed. \end{rem}

\section{Quasi-projectivity of $\cM$}

\subsection{Proof of Theorem \ref{CMemb}}\label{s-proof}
In this section, we verify the criterion for quasi-projectivity embedding in Theorem \ref{NM}, which generalizes the classical Nakai-Moishezon criterion to the normal non-complete algebraic space $U$ with a  compactification $M$. Theorem \ref{NM} follows from \cite{Nak00, Bir13} when the underlying space $M$ is known to be projective. We reduce the case of normal algebraic space to this known case. We do not know whether this hold for general proper algebraic space.


 \begin{thm}\label{NM}\Red{Let $M$ be a {\em normal} proper algebraic space that} is of finite type over $\mathbb{C}$.
Let $L$ be a line bundle on $M$ and  $M^{\circ}\subset M$ an open subspace. We assume
$L^m\cdot Z\ge  0$ for any $m$-dimensional irreducible subspace and the strictly inequality holds for any $Z$ meets $M^{\circ}$.
Then for sufficiently large power $k$, $|L^k|$ induces a rational map which is an embedding restricting on $M^{\circ}$.
\end{thm}

\begin{proof} We first show that it suffices to prove that $L^{\otimes k}$ separate any two points in $M^\circ$ for sufficiently large $k$. In fact, if this is true, then we can blow up the indeterminacy ideal $I$ of \Red{the rational map induced by $|L^{\otimes k}|$} and then take a normalization to get $\mu:M'\to M$. Then $\mu^*(L^{\otimes \Red{k} })=L_1+E$ where \Red{ $L_1\ge 0$ and $E$} is base point on $M'$, with the induced morphism separate any two points on $\mu^{-1}(M^\circ)\cong M^\circ$. Then we know that for sufficiently large $k_1$,
$$|L^{\otimes kk_1}|=\mu^*|L^{\otimes kk_1}|\supset |L_1^{\otimes k_1}|+k_1E$$
embeds $M^\circ$ as $M$ is normal.


Since there is always a Galois finite surjective morphism $f:M_1\to M$ from a normal scheme $M_1$ (cf. \cite[Lemma 2.8]{Kollar90}). Using the  Norm map, one easily see that
$f^*L$ separate any two points on $f^{-1}(M^\circ)$ implies
$${\rm Nm} |f^*(L^{\otimes k })|\subset |L^{\otimes k\cdot \deg f}|$$
separate two points on $M^\circ$ for  $k\gg 0$.
 So we can assume $M$ is a normal proper (possibly non-projective) scheme.

For any point $x\in M^\circ$, there is quasi-projective neighborhood  $U_x\subset M^\circ$ with $U_x$ being an open set of a projective scheme $M_x$. Consider the rational map  $M_x\dasharrow M$ and applying \cite[5.7.11]{RG71} to the morphism $\Gamma_x\to M$ from its graph $\Gamma_x$ to $M$, we see that the indeterminacy locus of $M_x\dasharrow M$ can be resolved by a sequence of blow ups. we know that there exists a normal variety $M'$ which admits morphisms $p:M'\to M$ and $q:M'\to M_x$ such that $q$ is relative projective over $M_x$. In particular, $M'$ is projective.

Consider $p^*L$ and the open set $U'_x\subset M'$ which is isomorphic to $U_x$, then the triple $(M', U_x', p^*L)$ satisfies the same assumption as $(M,M^\circ,L)$ in the theorem. Since $M'$ is projective, by our assumption, we know that $U'_x$ does not meet {the exceptional locus} $\mathbb{E}(p^*L)$, which is the union of subvarieties on which $p^*L$ is not big. Then it follows from \cite[Theorem 1.3]{Bir13} that for sufficiently large $k$, $p^*L^{\otimes k}$ does not have base points along  $U'_x$. Since $M$ is normal, this implies $L^{\otimes k}$ does not have base point along $U_x$.
\end{proof}

\begin{rem}In \cite[Theorem 6]{ST04}, a similar quasi-projectivity criterion was given in analytic setting.
\end{rem}

\begin{proof}[Proof of Theorem \ref{CMemb}]
Since the deformation of  any smooth Fano manifold is unobstructed, the Artin stack  classifying  $n$-dimensional smooth Fano manifolds   is  {\em smooth}. Now we consider the open substack parametrizing K-semistable Fano manifolds. Hence its {\em good moduli }space  in the sense of \cite{Alp13}  is  a normal algebraic space. So we  can apply Theorem \ref{NM} to the pair $(M,L):=(\overline\cM^{\rm n},n^*\Lam_\CM^{\otimes k})$ with $k$ being fixed in Lemma \ref{uni-k}
and $M^\circ=\cM$. For that, we need to show that for any irreducible subspace $Y\subset \overline\cM$ satisfying $Y\cap \cM\neq \emptyset$ we have $L^m\cdot Y>0$.



To achieve that, without loss of generality, we may assume that $Y$ is reduced.  We claim that there is a point $[z]\in V_{[z]}\subset Y\cap \cM$ together with an open neighborhood $V_{[z]}$, on which $(V_{[z]},\omega_\WP|_{V_{[z]}})$ is a smooth  K\"ahler manifold, from which we deduce $L^m\cdot Y=\int_Y\omega_\WP^m>0$ (cf. Proposition \ref{GE} ) and hence finish our proof.

To do that, let  us take a smooth point $[z']\in Y$ and
$$Z^\ast\supset\cU_{z'}\stackrel{f_{z'}}{\longrightarrow} \cV_{[z']}\subset \cM$$
 be the local quotient presentation as in Section \ref{descend}.
\Red{ Let $Z^\ast_{Y,\cU_{z'}}$ be a component of $f_{z'}^{-1} (\cV_{[z']}\cap Y)$ which dominates $(\cV_{[z']}\cap Y)$.
 Then by our construction
 $$f_{z'}|_{Z^\ast_{Y,\cU_{z'}}}:Z^\ast_{Y,\cU_{z'}}\longrightarrow \cV_{[z']}\cap Y$$
 is surjective, hence there is smooth point $z\in  Z^\ast_{Y,\cU_{z'}}$ such that $df_{z'}(z)$ is surjective by Bertini-Sard's theorem.  In particular, we are able to find a local slice  $z\in S$  of equal dimension locally isomorphic to an open neighborhood $V_{[z]}\subset S$ of $[z]=f_{z'}(z)\in V_{[z]}\subset\cV_{[z']}$.  So the slice $S$ must be transversal to $\aut(\cX_{z'})$-orbits near $z$, this implies that  the restriction of the universal family $\cX|_S\to S$ is {\em generically effective} of $(f _{z'}|_S)^{-1}(V_{[z]})$ by Kuranshi's local completeness Theorem. By Section \ref{WP-CM}, we conclude that the restriction of  {$\omega_\WP$  to a dense  open subset of $V_{[z]}$ is a {\em smooth K\"ahler form}}.
}\end{proof}

\subsection{Remarks on the projectivity of $\overline{\mM}$}\label{secproj}

We expect that the proper moduli space $\overline{\mM}$ constructed in \cite{LWX} is actually projective. Indeed, the CM line bundle $\Lambda_{\rm CM}$ can very well be ample (not only nef and big). Using Nakai-Moishezon's criterion for proper algebraic spaces \cite[Theorem 3.11]{Kollar90}, we just need to verify the positivity of intersection number $L^k\cdot Z$ for any subvariety $Z$ contained in $\overline{\cM}\setminus \cM$.
%
%
Using the notation as before, we just need to verify that the curvature of $h_{\rm DP}$ is strictly positive over an open set of the base $Z$.

Here we verify that this holds if $Z$ parametrizes Fano varieties with orbifold singularities. By restricting to an open subset of $Z$, we can assume that there is an {\em effective} flat family of Fano varieties denoted again by $\pi: \cX\rightarrow Z$, such that $\pi$ is a fibration with diffeomorphic orbifold fibers, i.e. $\cX_t$ is diffeomorphic to $\cX_{t'}$ as smooth orbifolds for any $t, t'\in Z$. In particular, the Kodaira-Spencer class comes from $H^1(\cX_t, \mathcal{T}^{\rm orb} )$. Here for any point $x$, there is an open neighborhood $U_x$ and a uniformization covering $\Pi_x: \widetilde{U}_x\rightarrow U_x$ such that $U_x=\widetilde{U}_x/G_x$ for a finite group $G_x$. Then $\mathcal{T}^{\rm orb}(U_x)$ is defined to be $\mathcal{T}(\widetilde{U}_x)^{G_x}$.
It is now well known that any weak K\"{a}hler-Einstein metric on a Fano variety with orbifold singularities is a smooth orbifold K\"{a}hler-Einstein metric. Furthermore we can assume that $\omega_t$ is a smooth family of orbifold K\"{a}hler-Einstein metrics by a straight-forward generalization of the results in \cite[Section 5.3]{Don2008} and \cite{Sz2010} to the orbifold setting.
On the other hand, by pulling back to local uniformization covering $\widetilde{U}_x$ it's easy to see that both $\widetilde{\omega}$ and $\tO$ in Section \ref{bypartialC0} are orbifold smooth. From the equation \eqref{KEeq2}, we also see that $\ku$ is an orbifold smooth function on $\cX$. Now 
we can do the following calculations. For simplicity, let us assume $Z$ is of complex dimension 1 with  a local coordinate function  $t$.
\begin{enumerate}
\item  Using the Stokes formula for fiber integrals along orbifold smooth fibers, $\partial\bar{\partial}$ and $\int_{\cX/S}$ can be interchanged.
\begin{eqnarray*}
\pdJd\kU&=&\pdJd\left(-\sum_{j=0}^n \int_{\cX/S} \tku \left(\widetilde{\omega}+\pdJd\tku\right)^{j}\wedge\widetilde{\omega}^{n-j}\right)\\
&=&-\sum_{j=0}^n \int_{\cX/S}\pdJd\tku  \wedge \left(\widetilde{\omega}+\pdJd\tku\right)^{j}\wedge\widetilde{\omega}^{n-j}.
\end{eqnarray*}
By the proof of Theorem \ref{exthDP}, we know that locally $h_{\rm DP}=e^{-\Psi_\alpha}=e^{-\widetilde{\Psi}_\alpha}e^{-\kU}$. So using \eqref{KrefS} we have:
\begin{eqnarray*}
\pdJd \Psi_\alpha&=&\pdJd \widetilde{\Psi}_\alpha+\pdJd\kU\\
&=&-\int_{\cX/S}\widetilde{\omega}^{n+1}-\sum_{j=0}^n \int_{\cX/S}\pdJd\tku  \wedge \left(\widetilde{\omega}+\pdJd\tku\right)^{j}\wedge\widetilde{\omega}^{n-j}\\
&=&-\int_{\cX/S}\omega^{n+1}.
\end{eqnarray*}
\item
$\omega$ is orbifold smooth and locally is equal to $\sddb\psi$. We can write:
\[
\omega^{n+1}=c(\psi)\omega_t^{n}\wedge \frac{\sqrt{-1}}{2\pi} dt\wedge d\bar{t},
\]
where $c(\psi)$ essentially measures the negativity of $\omega$ in the horizontal direction. By the same calculation as in \cite[Proposition 3]{Sch12} (see also \cite{Be10}), we see that $c(\psi)$ satisfies an elliptic equation:
\[
-\Delta_t\; c(\psi)-c(\psi)=|A|_{\omega_t}^2.
\]
Here $\Delta_t$ is the Laplace operator associated to the orbifold K\"{a}hler-Einstein metric $\omega_t$ on $\cX_t$ and $A\in A^{0,1}(\mathcal{T}^{\rm orb}_{\cX_t})$ represents the Kodaira-Spencer class of the deformation, which is obtained as follows. We choose local orbifold holomorphic coordinate $\{z^i, t\}$ on $\cX$. By the non-degeneracy of $\omega$ along the fiber $\cX_t$, there is a unique horizontal lifting $V$ of $\partial_t$ satisfying:
\[
d\pi(V)=\partial_t, \quad \omega(V, \partial_{\bar{z}^j})=0, \forall 1\le j\le n.
\]
Then $A=A_{\bar{j}}^i  d\bar{z}^j\otimes \partial_{z^i}$ is given by $(\bar{\partial} V)|_{\cX_t}$. For details, see \cite{Sch12}.

\item So we have:
\begin{eqnarray}\label{calA}
-\int_{\cX/S}\omega^{n+1}&=&-\int_{\cX/S}c(\psi)\omega_t^n\wedge \frac{\sqrt{-1}}{2\pi}dt\wedge d\bar{t}\nonumber\\
&=&-\int_{\cX/S}\left(\Delta_t c(\psi)+c(\psi)\right)\omega_t^n\wedge \frac{\sqrt{-1}}{2\pi} dt\wedge d\bar{t}\nonumber\\
&=&\left(\int_{\cX/S}|A|_{\omega_t}^2\omega_t^n\right)\frac{\sqrt{-1}}{2\pi} dt\wedge d\bar{t}.
\end{eqnarray}
\end{enumerate}
Since $\cX\rightarrow Z$ is generically effective, we know that $[A]\in H^1(\cX_t, \mathcal{T}^{\rm orb})$ is generically nonzero over $Z$. So for generic $t\in Z$, $A$ is non-vanishing over $\cX_t$ and the right-hand-side of \eqref{calA} is strictly positive. Using this strict positivity and similar arguments as in the proof of the quasi-projectivity of $\cM$, we get the following result:
\begin{prop}
Let $\overline{\cM}^{\rm orb}$ be the locus parametrizing smoothable K-polystable Fano varieties with at worst orbifold singularities. Then the normalization of $\overline{\cM}^{\rm orb}$ is quasi-projective.
\end{prop}
As a direct consequence, if we consider the case of del Pezzo surfaces, in which  we know that $\overline{\cM}=\overline{\cM}^{\rm orb}$ by \cite{Tian1990}. By applying the above strict positivity and Nakai-Moishezon's criterion for proper algebraic spaces (\cite{Kollar90}), we  immediately obtain:
\begin{cor}
The proper moduli spaces of smoothable K-polystable del Pezzo surfaces are projective.
\end{cor}
As mentioned before, this was known by \cite{OSS} except for del Pezzo surfaces of degree 1.

\section{Appendix I: A uniform convergence lemma}\label{append}

Let $\pi: \cX\rightarrow S$ be a flat family of klt Fano variety, which is {\em holomorphically }\Blue{embedded} into $\mathbb{P}^N\times S$. Assume that $m$ is chosen \Blue{in such a way} that $-m K_{\cX/S}$ is relatively base-point-free. Denote by $\{\widetilde{s}_i, 1\le i\le N_m\}$ the (holomophic) basis of the $\mathcal{O}_S$ module $\pi_*\mathcal{O}_{\cX}(-m K_{\cX/S})$. For any $t\in S$, denote $\widetilde{s}_i(t)=\widetilde{s}_i|_{\cX_t}$.  We can define a volume form on $\cX_t$ by
\[
\tO_t=\left(\sum_{i=1}^{N_m}|\Blue{\widetilde{s}_i}(t)|^2\right)^{-1/m}.
\]
The main technical lemma is
\begin{lem}\label{kltconv}
\Blue{In the above setting, we have the following uniform convergence:}
\[
\lim_{t\rightarrow 0} \int_{\cX_t}\tO_{t}=\int_{\cX_0}\tO_0.
\]
\end{lem}
We make some remarks before proving this convergence. In \cite{Li13} this Lemma was proved under the assumption that the generic fibre is smooth and $\dim S=1$. Here we generalize the calculations there to the general situation. In \cite{Li13} the simpler case of Lemma \ref{kltconv} was proved by lifting the integrals on both sides to a log-resolution of singularities $\pi:\widetilde{\cX}\rightarrow \cX$ and calculating carefully under the normal crossing coordinates. There it was proved that the limit of the left hand side of \eqref{lvolconv} as $t\rightarrow 0$ concentrates on the strict transform of $\cX_0$ under $\pi$ which coincides with the right hand side of \eqref{lvolconv}. This concentration phenomenon essentially only depends on a fundamental result in birational algebraic geometry: {\em inversion of adjunction}, which says that, in the ${\rm dim}_{\mathbb{C}}S=1$ case,  if $\cX_0$ is Kawamata-log-terminal (klt), then the pair $(\cX, \cX_0)$ is purely-log-terminal (plt) in a neighborhood of $\cX_0$. The klt property holds in our situation by \cite{BBEGZ,DS2012} because each $(\cX_t, \omega_t)$ is a K\"{a}hler-Einstein Fano variety. The plt property is expressed in terms of $a(\cX, \cX_0; E)>-1$ for any exceptional divisor $E$ of $\pi$ which does not have center on $\cX_0$. It's well known
that this lower bound of discrepancy (or complex exponent) implies an integrability condition, which turns out to be enough for us to apply dominant convergence theorem on the log-resolution $\widetilde{\cX}$ to get uniform integrability and confirm the convergence in \eqref{lvolconv}. 
\begin{rem}
The calculation of a similar kind was first carried out in \cite{Be12} and was then sharpened in \cite{Li13}. \Blue{ Indeed, a related continuity of Ding energy was speculated in \cite{Be12} and its importance was pointed out to the first author by Berman \cite{Be13}. }
\end{rem}

Here we use the similar arguments to deal with the higher codimensional case. We need to use a form of inversion of adjunction for higher codimensional klt subvariety \eqref{distotal}-\eqref{ivacoef}. Moreover, we need to use
the existence of toroidal reduction of family $\cX\rightarrow S$ proved by Abramovich-Karu (\cite{AK00}) to replace the role of log resolution in $\dim_{\mathbb{C}}S=1$ case.

Let's start by applying the toroidal reduction constructed in \cite{AK00} to obtain the following commutative diagram:
\begin{equation}\label{commut}
\xymatrix{
  \cY\ar[dr]_{\pi_{\mY}} \ar[r]^{\mu} &\cX\times_{S}T \ar[d]^{\pi} \ar[r]^{m_{\cX}} &
                    \cX \ar[d]_{\pi} \\
  & T  \ar[r]^{m_S} & S      
  }
\end{equation}
such that \Blue{$\cY$ and $T$ admit toroidal structures}, $\mu$ and $m_S$ are birational morphisms, and $\pi_\mY$ is a \Blue{flat} {\em toroidal} map.  It's clear that we just need to verify \eqref{kltconv} for $\cY\rightarrow T$, since $\cX\times_ST\rightarrow \Blue{T}$ is still a flat family of klt Fano varieties. So without of loss of generality we assume $T=S$ and fix a point $0\in S$ \Blue{from} now on. We will lift the calculation of integrals and limits to the space $\mY\rightarrow T=S$.

\Blue{Assume $\dim T=d$ such that $\dim \cY=n+d$}. Choose general hyperplane divisors $L_k$ on $T$ and $H_k$ its pull back on $\cX$ so that $\cX_0=\bigcap_{k=1}^{\Blue{d}} H_k$. Since $\mY_0$ and $\cX_0$ have the same dimension and $\mY_0\to \cX_0$ has connected fibers, we know $\mY_0$ has a component $\cX'_0$ which is  the strict transform of $\cX_0$ under $\mu$. Furthermore, the components of $\bigcap_{k=1}^d H_k'$ are normal and yields the log canonical centers of the sub-lc pair $(\mY, -K_{\mY/\cX}+\mu^*(\sum H^d_{k=1}))$  where $H_k'$ is the strict transform of $H_k$, then we indeed know that $\bigcap_{k=1}^d H_k'$ is irreducible because $\cX_0$ is a minimal log canonical center of $(\cX, \sum^d_{k=1}H_{k} )$.

We will need the following equalities defining the multiplicities denoted by $a_{ki}$:
\begin{equation}\label{pullHk}
\mu^* H_k =\pi_{\mY}^* L_k= H_k'+\sum_{i=1}^{\Blue{I}} a_{ki} E_i.
\end{equation}
Now comes the key ingredient. We write down the identity defining the discrepancies:
\begin{equation}\label{distotal}
K_{\cY/S}+\sum_{k=1}^{\Blue{d}} H'_k=\mu^*(K_{\cX/S}+\sum_{k=1}^{\Blue{d}} H_k)-\sum_{i=1}^{I} b_i E_i-\sum_{j=1}^{J} c_j F_j.
\end{equation}
where $E_i$ are vertical exceptional divisors and $F_j$ are horizontal exceptional divisors. Because $\cX_0$ is klt, by {\em inversion of adjunction}, we have:
\begin{equation}\label{ivacoef}
b_i<1, \text{ for } 1\le i\le I; \quad c_j<1, \text{ for } 1\le j\le J.
\end{equation}
Combining \eqref{distotal} and \eqref{pullHk} we also get:
\begin{equation}\label{distotal2}
K_{\cY/S}=\mu^*K_{\mX/S}-\sum_{i=1}^I \left(b_i-\sum_{k=1}^{d}a_{ki}\right)E_i-\sum_{j=\Blue{1}}^{\Blue{J}} c_j F_j.
\end{equation}
In the following, we will denote:
\[
a_i=\sum_{k=1}^d a_{ki}, \quad 1\le i\le I.
\]
As explained in \cite{Li13}, using the partition of unity argument, it's enough to show the following local convergence properties:
\begin{equation}\label{localconv}
\lim_{t\rightarrow 0}\int_{\cY\cap \mU(p, \delta)}\mu^*(v\wedge\bar{v})^{1/m}=\int_{\cY\cap \mU(p, \delta)} \mu|_{\cY}^*(v\wedge\bar{v})^{1/m}.
\end{equation}
where $p$ is any point in $\cY_0$, $\mU(p, \delta)$ is a small neighborhood of $p$ inside $\cY$ and $v$ is a local generator of $\mathcal{O}_{\cX}(-mK_{\cX/S})(\mU(p,\delta))$.

We will generalize the calculations as in \cite[Section 4]{Li13} to verify \eqref{localconv}. For any point
$p\in \cY_0$, there are 2 possibilities:
\begin{enumerate}
\item $p\in \cX_0'\cap \bigcap_{i=1}^{N_v} E_i\cap\bigcap_{j=1}^{N_h} F_j$.
\item $p\in \left(\bigcap_{i=1}^{N_v} E_i\cap \bigcap_{j=1}^{N_h} F_j\right)\setminus \cX_0'$.
\end{enumerate}
$N_v=N_v(p)$ (resp. $N_h=N_h(p)$) is the number of vertical (resp. horizontal) exceptional divisors passing through $p$. So if $N_v=0$ (resp. $N_h=0$), then there are no vertical (resp. horizontal) exceptional divisor passing through $p$ and the corresponding intersection does not appear.  
\begin{enumerate}
\item {\bf Case (1)}: Using the toroidal property of the map $\pi_{\cY}$,
we can choose local coordinates $\{(x_1, \dots, x_d; y_1,\dots, y_n)\}=:\{x, y\}$ {which are regular functions on $\cY$} near $p$ and $\{t_1,\dots,t_d\}$ near $0\in S$ such that
\begin{itemize}
\item
$x(p)=y(p)=0$,
\item locally $L_k=\{t_k=0\}$,
$H_ {k}'=\{x_k=0\} (1\le k\le d)$, $E_j=\{y_{\Red {j}}=0\} (1\le i\le N_v)$ and $F_j=\{y_{j}=0\} (N_v+1\le j\le N_v+N_h)$.
\item Since the pull back of $L_k$ is the sum of the reduced divisor $H_{k}'=0$ and other components, by \eqref{pullHk} the map $\pi_{\cY}$ is locally given as
\begin{equation}\label{xtot}
\begin{array}{lcl}
t_1&=& g_1(x,y)\cdot x_1\prod_{i=1}^{N_v}y_i^{a_{1i}}\\
&\dots& \\
t_{d}&=& g_d(x,y)\cdot x_d \prod_{i=1}^{N_v}y_i^{a_{di}}
\end{array}
\end{equation}
\end{itemize}
where $g_i(x,y) (1\le i\le d)$ are non vanishing holomorphic functions. For the simplicity of notations, we will assume $g_i(x,y)=1$ since it will be easy to modify the calculation for general non vanishing
$g_i(x,y)$.

{Now notice that the space $\mathcal{Y}$ in general has toric singularities which are good enough for us to carry out the local calculations by locally lifting to finite covers. So possibly by passing to finite covers, } let's consider the polydisk region:
\[
\mU(p, \delta)=\{|x_i|\le \delta, |y_j|\le\delta;\; 1\le i\le d, 1\le j\le n \}.
\]
When $t_i\neq 0\; (1\le i\le d)$, we can choose $\{y_1, \dots, y_n\}$ as the local coordinate system on the local fibre $\mU_t(p, \delta)=\mU(p,\delta)\cap \cY_t$:
\begin{equation}\label{txtox0}
x_i=x_i(t, y_1, \dots, y_n)=\frac{t_i}{\prod_{j=1}^{N_v} y_j^{a_{ij}}}, \quad 1\le i\le d.
\end{equation}
So when $t_i\neq 0\; (1\le i\le d) $, $\mU_t(p, \delta)$ is biholomorphic to the following region in the $y$-space via the projection:
\begin{equation}
\mV_t(\delta):=\left\{y=(y_1, \dots, y_n); \; |y_j|\le \delta, 1\le j\le n, \prod_{j=1}^{N_v} |y_j|^{a_{ij}}\ge |t_i|\delta^{-1}, 1\le i\le d\right\}\\
\end{equation}
Note that $\{\mV_t(\delta)\}$ is an increasing sequence of sets on the $y$-space with respect to the variable $t$. The limit is:
\[
\lim_{t\rightarrow 0}\mV_t(\delta)=\{y=(y_1, \dots, y_n)\in \mathbb{C}^n; \; |y_j|\le \delta, j=1, \dots, n\}=: \mV_0(\delta).
\]
Now choose a local generator $v=\{v_t\}$ of $mK_{\cX/\mathbb{S}}$ near $q=\mu(p)$.
\begin{equation}\label{pbv}
\mu^*\left(v^{1/m}\right)=g(x, y)\prod_{i=1}^{N_v}y_i^{a_i-b_i}\prod_{j=N_v}^{N_v+N_h}y_{j}^{-c_j} (dx\wedge dy\otimes \partial_t).
\end{equation}
Taking adjunction's $m$-times, we get:
\begin{equation}\label{algadj}
K_{\cX_0'}=\mu|_{\cX_0'}^* K_{\cX_0}-\sum_{i=1}^{\Blue{I}} b_i E_i|_{\cX_0'}-\sum_{j=1}^{\Blue{J}} c_j F_j|_{\cX_0'}.
\end{equation}
It will be useful for us to see this adjunction analytically. We will denote
\[
\cZ_k=\bigcap_{1\le l\le k} H'_l=\{ {x_l=0}\} \text{ for } k=1,\dots, d.
\]
Then
\[
\cX_0'=\cZ_d\subset \Blue{\cZ_{d-1}}\subset\dots \subset \cZ_1.
\]
By \eqref{xtot} and \eqref{pbv}, we get:
\begin{eqnarray*}
\mu|_{\cZ_1}^*(v^{1/m})&=&\left.g(x,y)\prod_{i=1}^{N_v} y_i^{a_i-b_i}\prod_{j=N_v+1}^{N_v+N_h} y_j^{-c_j} \frac{dt_1}{\prod_{i=1}^{N_v}y_i^{a_{1i}}}
\bigwedge_{k=2}^{m} dx_k \Blue{\wedge dy} \otimes \partial_{t_1}\bigwedge_{k=2}^m  \partial_{t_k}\right|_{\cZ_1}\\
&=&g(0,x_2,\dots, x_m, y)\prod_{i=1}^{N_v} y_i^{(a_i-a_{1i})-b_i}\prod_{j=N_v+1}^{N_v+N_h} y_j^{-c_j}\left(\bigwedge_{k=2}^m dx_k\Blue{\wedge dy}  \otimes \partial_{t_k}\right).
\end{eqnarray*}
Inductively, we indeed get the analytic formula corresponding to \eqref{algadj}:
\[
\mu|_{\cX_0'}(v^{1/m})=g({\bf 0}, y)\prod_{i=1}^{N_v} y_i^{-b_i}\prod_{j=N_v+1}^{N_v+N_h} y_j^{-c_j} dy.
\]
For the same reasons,  the local volume form along the fibre in \eqref{pbv} restricted $\mU_t$ becomes:
\begin{eqnarray*}
\left.\mu^*\left(v^{1/m}\right)\right|_{\mU_t}
&=&g(x(t,y), y)\prod_{i=1}^{N_v} y_i^{-b_i} \prod_{j=N_v+1}^{N_v+N_h} y_{j}^{-c_j}  dy.
\end{eqnarray*}
So
\[
\mu^*(v\wedge\bar{v})^{1/m}=\pm |g(x(t,y),y)|^2\prod_{i=1}^{N_v}|y_i|^{-2b_i}\prod_{j=N_v+1}^{N_v+N_h}|y_j|^{-2c_j}\cdot dy\wedge d\bar{y}.
\]
By \eqref{txtox0}, $\lim_{t\rightarrow 0}x(t, y)={\bf 0}$. So we see that for any $y\in \mV_t(\delta)$, we have:
\begin{eqnarray*}
\lim_{t\rightarrow 0}\mu^*(v\wedge \bar{v})^{1/m}&=&\pm |g({\bf 0}, y)|^2\prod_{i=1}^{N_v}|y_i|^{-2b_i}\prod_{j=N_v+1}^{N_v+N_h}|y_j|^{-2c_j}\cdot dy\wedge d\bar{y}\\\
&=&\mu|_{\cX_0'}^*(v_0\wedge \bar{v}_0)^{1/m}.
\end{eqnarray*}
Now it's straightforward to use the dominant convergence theorem to verify that (see \cite[(43)]{Li13}):
\begin{eqnarray*}
\lim_{t\rightarrow 0}\int_{\mU(p, \delta)\cap\widetilde{\cX}_t}\mu^*(v\wedge\bar{v})^{1/m}&=&\pm\lim_{t\rightarrow 0}\int_{\mV_t(\delta)}\frac{|g(x(t,y),y)|^2}{\prod_{i=1}^{N_v}|y_i|^{2b_i}\prod_{j=N_v+1}^{N_v+N_h}|y_j|^{2c_j}} dy\wedge d\bar{y}\\
&=&\pm\int_{\mV_0(\delta)}\frac{|g({\bf 0},y)|^2}{\prod_{i=1}^{N_v}|y_i|^{2b_i}\prod_{j=N_v+1}^{N_v+N_h}|y_j|^{2c_j}} dy\wedge d\bar{y}\\
&=&\int_{\tcX_0\cap \mU(p,\delta)}\mu|^*_{\cX_0'}(v\wedge\bar{v})^{1/m}.
\end{eqnarray*}
Notice that here we need to use the crucial fact from \eqref{ivacoef} that $b_i<1$ and $c_j<1$, which follows from the {\em inversion of adjunction}.

\item  {\bf Case 2: }
There are sub cases: The number of $H_i's$ containing $p$ is equal to $l$ for some $0\le l\le d-1$. In each sub case, we can choose coordinates $\{ x_{1}, \dots, x_{l}; y_1, \dots, y_{n+d-l})=:(x,y)$ on $\cY$ and $(t_1,\dots, t_{\Blue{d}})=:t$ on $S$ such that the toroidal map $\pi_{\cY}$ is defined by:
\begin{equation}\label{subcasel}
\begin{array}{lcllcl}
t_1&=& x_1 &\prod_{i=1}^{N_v}y_i^{a_{1i}} &\cdot& g_1(x,y);\\
&\dots& &\\
t_l&=&  x_l &\prod_{i=1}^{N_v}y_i^{a_{li}}&\cdot& g_l(x,y);\\
t_{l+1}&=& &\prod_{i=1}^{N_v} y_i^{a_{(l+1)i}}&\cdot& g_{l+1}(x,y); \\
&\dots& &\\
t_{d}&=& &\prod_{i=1}^{N_v}y_i^{a_{di}}&\cdot& g_{d}(x,y).
\end{array}
\end{equation}
Here $g_i(x,y)$ are non vanishing holomorphic functions. As before, we only deal with the case when $g_i(x,y)\equiv 1$ since the modification to the general case will be straightforward.
Also we will only consider the extremal case: $l=0$, because it will be clear that the other cases are mixture of Case (1) and this extremal case. So in the following we assume the following equalities hold:
\begin{equation}\label{casextremal}
\begin{array}{lccl}
t_1&=& &\prod_{i=1}^{N_v}y_i^{a_{1i}}\\
&\dots& & \\
t_{d}&=& &\prod_{i=1}^{N_v}y_i^{a_{di}}
\end{array}
\end{equation}
Since the map $\pi_{\mY}$ is dominant, we know that the matrix $(a_{ki})$ is of rank $d$.
\Blue{Consider again the following polydisk region by passing to local toric covers}:
\[
\mU(p, \delta)=\{|y_i|\le \delta,  1\le i\le n+d\}.
\]
We will show that the integral over $\mU_t=\mU\cap \cY_t$ converges to 0 as $t\rightarrow 0$.
Similar as in \cite{Li13}, it will be convenient to use the logarithmic coordinates. So we denote $t_k=e^{s_k}e^{\sqrt{-1}\phi_k}=e^{\tau_k}$, $y_i=e^{u_j}e^{\sqrt{-1}\theta_j}=e^{w_i}$, and \eqref{casextremal} becomes
\begin{equation}
s_k=\sum_{i=1}^{N_v} a_{ki} u_i, \quad 1\le k\le d.
\end{equation}
Then it's easy see that we have
\[
\mU_t(p, \delta)\cong \underline{\mV}_{t}(\delta)\times (S^1)^{N_v-d}\times \left\{|y_j|\le\delta, N_v+1\le j\le n+d \right\},
\]
where the first factor on the right is a {\em bounded} polytope:
\[
\underline{\mV}_s:= \underline{\mV}_{s}(p, \delta)=\left\{ u_i<-\log \delta, \sum_{i=1}^{N_v} a_{ki}u_i=s_k, 1\le k\le d \right\}\subset \mathbb{R}^{N_v}.
\]
For the pull-back of holomorphic form, we have the similar formula as in \eqref{pbv} which follows from \eqref{distotal2}:
\[
\mu^*(v^{1/m})=g(y)\prod_{i=1}^{N_v} y_i^{a_i-b_i}\prod_{j=N_v+1}^{N_v+N_h}y_j^{-c_j} (dy\otimes\partial_{t}).
\]
To transform into logarithmic coordinates, we use:
\[
y_i^{a_i-b_i}dy_i=e^{(1+a_i-b_i)w_i}dw_i, \quad \partial_{t_k}=\frac{\partial_{\tau_k}}{t_k}=\frac{\partial_{\tau_k}}{\prod_{j=1}^{N_v} y_j^{a_{kj}}}=\frac{\partial_{\tau_k}}{\prod_{j=1}^{N_v} e^{a_{kj}w_j}}.
\]
So we can get:
\[
\mu^*(v^{1/m})=\pm g(y)\left(\prod_{i=1}^{N_v} e^{(1-b_i)w_i} \bigwedge_{i=1}^{N_v}dw_i\otimes \partial_{\tau}\right) \bigwedge_{j=N_v+1}^{N_v+N_h} y_j^{-c_j} dy_j \wedge dy',
\]
\Blue{where $dy'=\bigwedge_{j=N_v+N_h+1}^{n+d}dy_j$.
So we have:
\begin{eqnarray*}
\mu^*(v\wedge \bar{v})^{1/m}&=&|g(y)|^2 \left(\bigwedge_{i=1}^{N_v}e^{2(1-b_i)u_i}du_i\otimes \bigwedge_{k=1}^d \partial_{s_k}\right)\left(\bigwedge_{i=1}^{N_v} d\theta_i\otimes\bigwedge_{k=1}^d \partial_{\phi_k}\right)\wedge \\
&&\left(\bigwedge_{j=N_v+1}^{N_v+N_h} |y_j|^{-2c_j}dy_j\wedge d\bar{y}_j\right)\wedge dy'\wedge d\overline{y'}.
\end{eqnarray*}
Notice that since $c_j<1$, $|y_j|^{-2c_j}dy_j\wedge d\overline{y}_j$ is integrable. }So we just need to estimate:
\begin{equation}\label{polest}
\int_{\underline{\mV}_s}\bigwedge_{i=1}^{N_v} e^{2(1-b_i)u_i}du_i\otimes\bigwedge_{k=1}^d \partial_{s_k}.
\end{equation}
Note that $\underline{\mV}_s$ is a $(N_v-d)$-dimensional polytope in $\mathbb{R}^{N_v}$ defined by linear functions. By co-area formula, we know that
\[
\bigwedge_{i=1}^{N_v} du_i=\frac{1}{A}{\rm dvol}\otimes \bigwedge_{k=1}^{\Blue{d}} d s_k.
\]
Here we have denoted by ${\rm dvol}$ the Euclidean volume form on $\underline{\mV}_s$, and $A=\det\left(\la{\bf a}_k, {\bf a}_l\ra\right)^{1/2}$, where
\[
{\bf a}_k=\{ a_{ki}\}=\nabla s_k, \quad \la {\bf a}_k, {\bf a}_l\ra=\sum_{i=1}^{N_v} a_{ki}a_{li}.
\]
So we see that the integral in \eqref{polest} is equal to:
\begin{equation}\label{coarea}
\frac{1}{A}\int_{\underline{\mV}_s}\prod_{i=1}^{N_v} e^{2(1-b_i)u_i} {\rm dvol}.
\end{equation}
Using $b_i<1$ for $1\le i\le N_v$, it's now an easy exercise to verify that the integral in \eqref{coarea} converges {\em uniformly} to 0 as $s\rightarrow -\infty$ (meaning $s_k \rightarrow -\infty$ uniformly with respect to $k$).

As mentioned before, for the general sub cases in \eqref{subcasel}, we can first use calculus of adjunction in Case (1) $l$-times to kill the variables $x_1, \dots, x_l$ and reduce to the extremal sub case in Case (2). So we know that the contribution in all sub cases of Case (2) are indeed 0 as $t\rightarrow 0$.
\end{enumerate}


\begin{bibdiv}
\begin{biblist}


\bib{AK00}{article}{
author={Abramovich, D.},
author={Karu, K.},
title={Weak semistable reduction in characteristic 0},
journal={Invent. math.},
volume={139},
year={2000},
pages={241-273},
}


\bib{AFSV14}{article}{
author={Alper, Jarod }
author={Fedorchuk, Maksym }
author={Smyth, David Ishii }
author={Van der Wyck, Frederick}
title={Log minimal model program for the moduli space of stable curves: the second flip}
journal={arXiv:1308.1148}
}

\bib{Alp13}{article}{
author={Alper, Jarod }
title={Good moduli spaces for Artin stacks}
journal={Ann. Inst. Fourier (Grenoble)}
date={2013}
volume={63}
number={6}
pages={2349-2042}}

\bib{Be10}{article}{
author={Berman, Robert J.},
title={Relative Kahler-Ricci flows and their quantization},
journal={arXiv:1002.3717},
year={2010},
}

\bib{Be12}{article}{
  author={Berman, Robert J.}
  title={K-polystability of $\QQ$-Fano varieties admitting Kahler-Einstein metrics}
  journal={ arXiv:1205.6214}
  date={2012}
}

\Blue{
\bib{Be13}{article}{
author={Berman, Robert J.},
title={Private communication},
year={2013},
}
}

\bib{BBEGZ}{article}{
  author={Berman, Robert J.}
  author={Boucksom, S\'ebstien}
  author={Eyssidieux,Philippe}
  author={Guedj, Vincent}
  author={Zariahi, Ahmed}
  title={K\"a hler-Einstein metrics and the KŠhler-Ricci flow on log Fano varieties}
  date={2011}
  journal={arXiv:1111.7158}
}


\bib{BG13}{article}{
  author={Berman, Robert J.}
   author={Guenancia, Henri}
  title={K\"ahler-Einstein metrics on stable varieties and log canonical pairs}
  journal={ arXiv:1304.2087, to appear in GAFA}
  date={2014}
}


\bib{Bir13}{article}{
author={Birkar, Caucher},
title={The augmented base locus of real divisors over arbitrary fields},
journal={arXiv:1312.0239},
year={2013},
}



\bib{CC97}{article}{
author={Cheeger, J.},
author={Colding, T.H.},
title={On the structure of spaces with Ricci curvature bounded below I},
journal={J. Diff. Geom.},
volume={45},
year={1997},
pages={1-75},
}

\bib{CDS1}{article}{
 author={Chen, Xiuxiong}
 author={Donaldson, Simon}
 author={Sun, Song}
 title={K\"ahler-Einstein metrics on Fano manifolds. I: Approximation of metrics with cone singularities.},
 journal={To appear in J. Amer. Math. Soc.  },
 volume={},
 number={},
 date={2014},
 pages={},
 issn={0003-486X}
 }


\bib{CDS2}{article}{
 author={Chen, Xiuxiong}
 author={Donaldson, Simon}
 author={Sun, Song}
 title={K\"ahler-Einstein metrics on Fano manifolds. II: Limits with cone angle less than $2\pi$.},
 journal={To appear in J. Amer. Math. Soc.  },
 volume={},
 number={},
 date={2014},
 pages={},
 issn={0003-486X}
 }


\bib{CDS3}{article}{
 author={Chen, Xiuxiong}
 author={Donaldson, Simon}
 author={Sun, Song}
 title={K\"ahler-Einstein metrics on Fano manifolds. III: Limits as cone angle approaches $2\pi$ and completion of the main proof.},
 journal={To appear in J. Amer. Math. Soc.  },
 volume={},
 number={},
 date={2014},
 pages={},
 issn={0003-486X}
 }

\bib{Del87}{article}{
author={Deligne, P.},
title={Le D\'{e}terminant de la Cohomolgie},
journal={Current Trends in Arithmetrical Algebraic Geometry, Contemp. Math.},
volume={67},
year={1987},
pages={93-177},
publisher={Amer. Math. Soc., Providence, RI},
}


\bib{Dem}{article}{
author={Demailly, J.P.},
title={Complex Analytic and Differential Geometry},
journal={http://www.fourier.ujf-grenoble.fr/demailly/manuscripts/agbook.pdf},
year={2012}
}



 \bib{Don00}{article}{
 author={Donaldson, Simon K.},
 title={Scalar curvature and stability of toric varieties.}
 journal={J. Differential Geom.},
 volume={62},
 year={2002},
 pages={289-349},
 }
	
\bib{Don01}{article}{
author={Donaldson, Simon K.},
title={Scalar curvature and projective embeddings, I,},
journal={J. Differential Geom.},
volume={59},
pages={479-522},
year={2001},
}

\bib{Don2008}{article}{
author={Donaldson, Simon K.},
title={K\"ahler geometry on toric manifolds, and some other manifolds with large symmetry.}
note={Adv. Lect. Math. (ALM), Handbook of geometric analysis. No. 1, Int. Press, Somerville, MA}
volume={7}
year={2008}
pages={29-75}
}



\bib{Don10}{article}{
author={Donaldson, Simon K.},
title={Stability, birational transformations and the K\"{a}hler-Einstein problem},
journal={Surveys in Differential Geometry, arXiv:1007.4220},
year={2010},
}

\bib{Don2013}{article}{
author={Donaldson, Simon K.},
title={Volume estimates, Chow invariants and moduli of K\"ahler-Einstein metrics}
journal={http://www.math.u-psud.fr/~repsurf/ERC/Bismutfest/Bismutfest.html}
date={2013}
}


\bib{DS2012}{article}{
   author={Donaldson, Simon K.},
   author={Sun, Song},
   title={Gromov-Hausdorff limits of K\"ahler manifolds and algebraic geometry},
   journal={Acta Math. },
   volume={213}
   number={1},
   date={2014},
   pages={63-106},
   issn={0073-8301},
   review={\MR{0498551 (58 \#16653a)}},
}



\bib{DN89}{article}{
author={Drezet, J.-M.},
author={Narasimhan, M. S.}
title={Groupe de Picard des vari\'et\'es de modules de fibr\'es semi-stables sur les courbes alg\'ebriques.}
journal={Invent. Math. }
volume={97}
number={1}
year={1989}
pages={53-94}
}

\bib{DT92}{article}{
author={Ding, W.},
author={Tian, G.},
title={K\"{a}hler-Einstein metrics and the generalized Futaki invariant},
journal={Invent. Math.},
volume={110},
number={2},
year={1992},
pages={315-335},
}

\bib{EGZ11}{article}{
author={Eyssidieux, P.},
author={Guedj, V.},
author={Zeriahi, A.},
title={viscosity solutions to degenerate complex Monge-Amp\`{e}re equations},
journal={Comm. Pure Appl. Math.},
volume={64},
year={2011},
number={8},
pages={1059-1094},
}

\bib{FN80}{article}{
author={Fornaess, J. E.},
author={Narasimhan, R.},
title={The Levi problem on complex spaces with singularities},
journal={Math. Ann.},
volume={248},
year={1980},
pages={47-72},
}


\bib{FS90}{article}{
author={Fujiki,A.},
author={Schumacher,G.},
title={The moduli space of extremal compact K\"{a}hler
manifolds and generalized Weil-Petersson metrics},
journal={Publ. Res. Inst. Math.},
volume={26},
year={1990},
pages={101-183},
}


\bib{Fuj12}{article}{
author={Fujino, O},
title={Semipositivity theorems for moduli problems},
journal={arXiv:1210.5784},
year={2012},
}

\bib{Fut83}{article}{
author={Futaki, A.},
title={An obstruction to the existence of Einstein K\"{a}hler metrics},
journal={Inventiones Mathematicae},
pages={437-443},
issue={73},
year={1983},
}

\bib{Fut90}{article}{
author={Futaki, A.},
title={K\"{a}hler-Einstein metrics and integral invariants},
journal={Lecture Notes in Mathematics},
number={1314},
year={1990},
publisher={Springer-Verlag}
}

\bib{Gra62}{article}{
author={Grauert, Hans},
title={\"{U}ber Modifikationen und exzeptionelle analytische Mengen},
journal={Math. Annalen},
volume={146},
year={1962},
pages={331-368},
}


\bib{GR56}{article}{
author={Grauert, H.},
author={Remmert, R.},
title={Plurisubharmonische Funktionen Mengen R\"{a}umen},
journal={Math. Z.},
volume={65},
year={1956},
pages={175-194},
}






\bib{Kollar90}{article}{
author={Koll{\'a}r, J{\'a}nos},
title={Projectivity of complete moduli},
journal={J. Differ. Geom.},
volume={32},
pages={235-268},
year={1990},
}



\bib{Kollar06}{article}{
author={Koll{\'a}r, J{\'a}nos},
title={non-quasi-projective moduli spaces},
journal={Ann. Math.},
volume={164},
pages={1077-1096},
year={2006},
}



\bib{Kollar13}{article}{
    AUTHOR = {Koll{\'a}r, J{\'a}nos},
     TITLE = {Moduli of varieties of general type},
 BOOKTITLE = {Handbook of moduli. {V}ol. {II}},
    SERIES = {Adv. Lect. Math. (ALM)},
    VOLUME = {25},
     PAGES = {131--157},
 PUBLISHER = {Int. Press, Somerville, MA},
      YEAR = {2013},
  }

\bib{KMM92}{article}{
    AUTHOR = {Koll{\'a}r, J{\'a}nos}
    author={Miyaoka, Yoichi}
    author={ Mori, Shigefumi},
     TITLE = {Rational connectedness and boundedness of {F}ano manifolds},
   JOURNAL = {J. Differential Geom.},
  FJOURNAL = {Journal of Differential Geometry},
    VOLUME = {36},
      YEAR = {1992},
    NUMBER = {3},
     PAGES = {765--779},
}


\bib{Li13}{article}{
author={Li, Chi},
title={Yau-Tian-Donaldson correspondence for K-semistable Fano manifolds},
journal={arXiv:1302.6681v5},
year={2013},
}

\bib{LWX}{article}{
author={Li, Chi},
author={Wang, Xiaowei},
author={Xu, Chenyang},
title={Degeneration of Fano Kahler-Einstein manifolds},
journal={arxiv:1411.0761},
year={2014},
}


\bib{Mor99}{article}{
author={Moriwaki, A.},
title={The continuity of Deligne's pairing},
journal={Internat. Math. Res. Notices},
number={19},
year={1999},
pages={1057-1066},
}

\bib{Nak00}{article}{
    AUTHOR = {Nakamaye, Michael},
     TITLE = {Stable base loci of linear series},
   JOURNAL = {Math. Ann.},
  FJOURNAL = {Mathematische Annalen},
    VOLUME = {318},
      YEAR = {2000},
    NUMBER = {4},
     PAGES = {837--847},
      ISSN = {0025-5831},
}





\bib{Odaka14a}{article}{
 author={Odaka, Yuji},
 title={On the moduli of K\"ahler-Einstein Fano manifolds},
 journal={ Proceeding of Kinosaki algebraic geometry symposium 2013, arXiv:1211.4833. },
 volume={},
 number={},
 date={2012},
 pages={},
 issn={}
 }

\bib{Odaka14b}{article}{
author={Odaka, Yuji},
title={Compact moduli space of K\"ahler-Einstein Fano varieties},
journal={arXiv:1412.2972},
year={2014}}


\bib{OSS}{article}{
 author={Odaka, Yuji}
  author={Spotti, Cristiano }
   author={Sun, Song}
 title={Compact Moduli Spaces of Del Pezzo Surfaces and K\"ahler-Einstein metrics},
 journal={arXiv:1210.0858 },
year={2012}
}

\bib{OSY}{article}{
    AUTHOR = {Ono, Hajime}
    author={Sano, Yuji}
    author={Yotsutani, Naoto},
     TITLE = {An example of an asymptotically {C}how unstable manifold with
              constant scalar curvature},
   JOURNAL = {Ann. Inst. Fourier (Grenoble)},
    VOLUME = {62},
      YEAR = {2012},
    NUMBER = {4},
     PAGES = {1265--1287},
   }


\bib{Pa2012}{article}{
 author={Paul, Sean}
 title={CM stability of projective varieties},
 journal={arXiv:1206.4923 },
 volume={},
 number={},
 date={2012},
 }

\bib{PS04}{article}{
author={Phong, D.H.},
author={Sturm, Jacob},
title={Scalar Curvature, moment maps, and the Deligne pairing},
journal={American Journal of Mathematics},
volume={126},
number={3},
year={2004},
pages={693-712},
}





\bib{PRS08}{article}{
author={Phong, D.H.},
author={Ross, J. },
author={Sturm, J},
title={Deligne pairing and the Knudsen-Mumford expansion},
journal={J. Differential Geom.},
volume={78},
year={2008},
number={ 3},
pages={475-496},
}


\bib{PT09}{article}{
author={Paul, S.},
author={Tian, G.},
title={CM stability and the
generalized Futaki invariant II},
journal={Ast\'{e}risque},
volume={328},
pages={339-354},
year={2009},
}

\bib{RG71}{article}{
    AUTHOR = {Raynaud, Michel}
    AUTHOR={Gruson, Laurent},
     TITLE = {Crit\`eres de platitude et de projectivit\'e. {T}echniques de
              ``platification'' d'un module},
   JOURNAL = {Invent. Math.},
  FJOURNAL = {Inventiones Mathematicae},
    VOLUME = {13},
      YEAR = {1971},
     PAGES = {1--89},
      ISSN = {0020-9910},
   MRCLASS = {14A15 (13C10)},
  MRNUMBER = {0308104 (46 \#7219)},
MRREVIEWER = {M. Maruyama},
}
	

\bib{Sch12}{article}{
author={Schumacher, Georg}
title={Positivity of relative canonical bundles and applications},
journal={Invent. Math.},
year={2012},
volume={190},
pages={1-56},
}

\bib{Sch13}{article}{
author={Schumacher, Georg}
title={Erratum: Positivity of relative canonical bundles and applications},
journal={Invent. Math.},
year={2013},
volume={192},
pages={253-255},
}




\bib{SSY}{article}{
author={Sun, Song},
author={Spotti, Cristiano},
author={Yao, Chengjian},
title={Existence and deformations of Kahler-Einstein metrics on smoothable $\QQ$-Fano varieties},
journal={arXiv:1411.1725},
year={2014},
}


\bib{ST04}{article}{
author={Schumacher, G.},
author={Tsuji, H.},
title={Quasi-projectivity of moduli spaces of polarized varieties},
journal={Ann. Math.},
volume={159},
year={2004},
pages={597-639},
}


\bib{Sz2010}{article}{
author={Sz\'{e}kelyhidi, Gabor}
title={The K\"ahler-Ricci flow and K-polystability}
volume={132}
number={4}
journal={Amer. J. Math. }
year={2010}
pages={1077-1090}
}

\bib{Yos07}{article}{
author={Yoshikawa, K.-I.},
title={On the singularity of Qullen metrics},
journal={Math. Ann.},
year={2007},
volume={337},
pages={61-89},
}


\bib{Tian1987}{article}{
author={Tian, Gang},
title={Smoothness of the universal deformation space of compact Calabi-Yau manifolds and
its Petersson-Weil metric},
journal={Mathematical Aspects of String Theory (ed. S.-T. Yau)},
publisher={World Scientific},
year={1987},
pages={629-646},
}


\bib{Tian1990}{article}{
author={Tian, Gang}
title={On Calabi's conjecture for complex surfaces with positive first Chern class.}
journal={Invent. Math.}
volume={101}
year={1990}
number={1}
pages={101-172}
}


\bib{Tian1997}{article}{
author={Tian, Gang},
title={K\"{a}hler-Einstein metrics with positive scalar curvature},
journal={Invent. Math.},
volume={130},
year={1997},
pages={1-39},
}

\bib{Tian2000}{article}{
author={Tian, Gang},
title={Bott-Chern forms and geometric stability},
journal={Discret Contin. Dynam. Systems},
Volume={6},
year={2000},
pages={211-220},
}

\bib{Tian2012}{article}{
author={Tian, Gang},
title={Existence of Einstein metrics on Fano manifolds},
journal={Metric and Differential Geometry,}
volume={297},
year={2012},
pages={119-159},
}


\bib{Tian2014}{article}{
author={Tian, Gang},
title={K-stability and K\"{a}hler-Einstein metrics},
journal={arXiv:1211.4669},
year={2012}
}


\bib{Tian13}{article}{
    AUTHOR = {Tian, Gang},
     TITLE = {Partial {$C^0$}-estimate for {K}\"ahler-{E}instein metrics},
   JOURNAL = {Commun. Math. Stat.},
  FJOURNAL = {Communications in Mathematics and Statistics},
    VOLUME = {1},
      YEAR = {2013},
    NUMBER = {2},
     PAGES = {105--113},
     }

  \bib{Tian14}{article}{
    AUTHOR = {Tian, Gang},
     TITLE = {K-stability implies CM-stability},
   JOURNAL = {arXiv:1409.7836},
   year={2014},
     }



 \bib{Var89}{article}{
    AUTHOR = {Varouchas, J.},
     TITLE = {K\"{a}hler spaces and proper open morphisms},
   JOURNAL = {Math. Ann.},
    VOLUME = {283},
      YEAR = {1989},
       PAGES = {13-52},
     }



\bib{Vie95}{book}{
  AUTHOR = {Viehweg, Eckart},
     TITLE = {Quasi-projective moduli for polarized manifolds},
    SERIES = {Ergebnisse der Mathematik und ihrer Grenzgebiete (3) [Results
              in Mathematics and Related Areas (3)]},
    VOLUME = {30},
 PUBLISHER = {Springer-Verlag, Berlin},
      YEAR = {1995},
 }

\bib{Zha96}{article}{
author={Zhang, Shouwu},
title={Heights and reductions of semi-stable varieties},
journal={Compos. Math. },
volume={104},
pages={77-105},
year={1996},
}




 \end{biblist}
 \end{bibdiv}
\end{document}